\documentclass[a4paper, 10pt]{amsart}
\usepackage[top=80pt,bottom=65pt,left=70pt,right=70pt]{geometry}
\usepackage{graphicx, eepic}
\usepackage{times}

\normalfont
\usepackage{bbm}
\usepackage{xcolor}
\usepackage{verbatim}
\usepackage{kotex}
\usepackage{mathrsfs}
\usepackage{amscd}
\usepackage{hhline}
\usepackage{float}
\usepackage[all]{xy}
\usepackage[subnum]{cases}
\ifpdf
\usepackage{tikz, tikz-cd}
\usepackage{tikz-3dplot}

\usepackage{mathdots}
\fi
\usetikzlibrary{arrows,matrix,positioning}

\usepackage{caption}
\usepackage[labelformat=simple]{subcaption}

\usepackage{makecell}

\usepackage{amsthm,amsmath,amsfonts,amssymb}
\usepackage{hyperref}
\hypersetup{colorlinks}
\hypersetup{
    bookmarks=true,         
    unicode=false,          
    pdftoolbar=true,        
    pdfmenubar=true,        
    pdffitwindow=false,     
    pdfstartview={FitH},    
    pdftitle={My title},    
    pdfauthor={Author},     
    pdfsubject={Subject},   
    pdfcreator={Creator},   
    pdfproducer={Producer}, 
    pdfkeywords={keyword1} {key2} {key3}, 
    pdfnewwindow=true,      
    colorlinks=true,       
    linkcolor=blue,          
    citecolor=red,        
    filecolor=violet,      
    urlcolor=violet       
}
\usepackage{multirow, url}
\usepackage{color}
\usepackage[all]{xy}
\theoremstyle{plain}
\newtheorem{theorem}{Theorem}[section]

\newtheorem{thmx}{Theorem}

\newtheorem{proposition}[theorem]{Proposition}
\newtheorem{lemma}[theorem]{Lemma}

\newtheorem{corollary}[theorem]{Corollary}

\theoremstyle{definition}
\newtheorem*{question}{Question}

\newtheorem{definition}[theorem]{Definition}

\newtheorem{example}[theorem]{Example}

\newtheorem{question_num}[theorem]{Question}

\theoremstyle{remark}
\newtheorem{remark}[theorem]{Remark}

\renewcommand{\hat}{\widehat}

\newcommand{\C}{\mathbb{C}}

\newcommand{\R}{\mathbb{R}}

\newcommand{\Z}{\mathbb{Z}}

\newcommand{\mcal}{\mathcal}

\newcommand{\node}{\mathrm{node}}
\newcommand{\SL}{\textup{SL}}
\newcommand{\Gr}{\textup{Gr}}

\def\mcal{\mathcal}
\def\frak{\mathfrak}
\def\scr{\mathscr}

\newcommand{\vs}{\vspace}
\newcommand{\hs}{\hspace}

\numberwithin{equation}{section} \numberwithin{table}{section}

\newcommand{\defi}[1]{{\textit{#1}}}
\newcommand{\xkj}[2]{{x_{#1,#2}}}
\newcommand{\tkj}[2]{{t_{#1,#2}}}
\newcommand{\skj}[2]{{s_{#1,#2}}}

\begin{document}                                                                          

\title{On the combinatorics of string polytopes}
\author{Yunhyung Cho}
\address{Department of Mathematics Education, Sungkyunkwan University, Seoul, Republic of Korea}
\email{yunhyung@skku.edu}
\author{Yoosik Kim}
\address{Department of Mathematics and Statistics, Boston University, Boston, MA, USA}
\email{kimyoosik27@gmail.com, yoosik@bu.edu}
\author{Eunjeong Lee}
\address{Center for Geometry and Physics, Institute for Basic Science (IBS), Pohang 37673, Korea}
\email{eunjeong.lee@ibs.re.kr}
\author{Kyeong-Dong Park}
\address{Center for Geometry and Physics, Institute for Basic Science (IBS), Pohang 37673, Korea}
\email{kdpark@ibs.re.kr}

\begin{abstract}
	For a reduced word ${\bf i}$ of the longest element in the Weyl group of $\SL_{n+1}(\C)$, one can associate the {\em string cone} $C_{\bf i}$
	which parametrizes the dual canonical bases.
	In this paper, we classify all ${\bf i}$'s such that $C_{\bf i}$ is simplicial. We also prove that for any 
	regular dominant weight $\lambda$ of $\frak{sl}_{n+1}(\C)$, the corresponding string polytope $\Delta_{\bf i}(\lambda)$ is unimodularly equivalent to the 
	Gelfand--Cetlin polytope associated to $\lambda$ if and only if $C_{\bf i}$ is simplicial. Thus we completely characterize Gelfand--Cetlin type string polytopes 
	in terms of ${\bf i}$.
\end{abstract}
\maketitle
\setcounter{tocdepth}{1} 
\tableofcontents

\section{Introduction}
\label{secIntroduction}

	The \emph{string polytope} was introduced by Littelmann \cite{Li} as a generalization of the Gelfand-Cetlin polytope (see~\cite{GC1950}).	
	For a connected reductive algebraic group $G$ over $\mathbb C$ and a dominant integral weight $\lambda$, a choice of a reduced word ${\bf i}$ of the longest element in the Weyl group of $G$ determines the string polytope $\Delta_{\bf i}(\lambda)$.
	Depending on a choice of string parameterizations, that is a choice ${\bf i}$ of words, combinatorially distinct string polytopes arise in general.	
	Littelmann \cite{Li} and Berenstein--Zelevinsky \cite{BZ} found a remarkably explicit description for the string polytopes reflecting a choice of words.
	
	In representation theory, the polytope has been provided a combinatorial understanding of crystal bases in finite-dimensional representations.
	In particular, the string polytope $\Delta_{\bf i}(\lambda)$ records parameterizations with respect to ${\bf i}$ of Lusztig's canonical bases \cite{Lu90} or equivalently of Kashiwara's global bases \cite{Kas90} for the irreducible representation of $G$ with highest weight $\lambda$ as its lattice points.
		
	In algebraic geometry, the string polytope $\Delta_{\bf i}(\lambda)$ has been exploited to study (partial) flag varieties and more generally spherical varieties.
	The dominant weight $\lambda$ assigns the line bundle over the flag variety associated to $\lambda$. 
	A choice of a non-zero global section and the reduced word $\bf i$ determines the valuation on the field of rational functions, which defines the Newton--Okounkov body. 
	The Newton--Okounkov body for the highest term valuation associated to the reduced word ${\bf i}$ indeed coincides with the string polytope $\Delta_{\bf i}(\lambda)$ (see \cite{FN17,Kav15}). 
	By works of \cite{An,  Cal,GL, KM}, the flag variety degenerates into the toric variety corresponding to the string polytope~$\Delta_{\bf i}(\lambda)$. It enables us to study flag and spherical varieties using the combinatorics of the string polytope $\Delta_{\bf i}(\lambda)$.
	
	Besides their importance in representation theory and algebraic geometry, understanding combinatorics of string polytopes is also important in various versions of mirror symmetry.
	In order to run Batyrev's construction for a mirror family of Calabi--Yau complete intersections in flag varieties, the reflexivity of a string polytope is crucial in \cite{AB, BCKV}.  
	A Minkowski decomposition of a string polytope is related to a construction of a Lagrangian torus fibration for Strominger--Yau--Zaslow mirror symmetry in \cite{Gross}.
	Moreover, the string polytopes provide an important class of examples of the foundational work of Gross--Hacking--Keel--Kontsevich \cite{GHKK}. 
		
	This paper is concerned with the string polytopes of Lie type $A$ with a dominant weight $\lambda$. 
	The string polytope $\Delta_{\bf i}(\lambda)$ is the intersection of two cones: the string cone $C_{\bf i}$ and the $\lambda$-cone $C_{\bf i}^\lambda$. 
	In this case where $G = \SL_{n+1}(\C)$, the inequalities defining the string cone $C_{\bf i}$ can be written from the Gleizer--Postnikov's paths in the wiring diagram (see \cite{GP}). 
	Also, Littelmann's description for the $\lambda$-cone can be read off from the wiring diagram in \cite{Ru}.
	The string polytopes are related by piecewise-linear transformations under the changes of reduced words as a geometric lifting (see \cite{BZ}). 
	In terms of cluster algebras invented by Fomin--Zelevinsky \cite{FZ}, the process has been understood as a (quiver) mutation in \cite{BF}.
	
	Despite the explicit expression and the rich combinatorial realization of the string polytopes, finding combinatorial properties of arbitrary dimensional string polytopes is quite challenging. 
	It is because observations on low dimensional examples are \emph{not} usually valid in the high dimensional case.
	The integrality of all string polytopes for $\SL_n(\mathbb C)$, $n \leq 5$, had been observed by Alexeev and Brion in \cite{AB}. They conjectured that the string polytope~$\Delta_{\bf i}(\lambda)$ of Lie type~$A$ is integral for any reduced word $\bf i$ of the longest element in the Weyl group and for any dominant integral weight $\lambda$. 
	However, it recently turns out that this conjecture does not hold because there is a reduced word $\bf i$ such that the string polytope $\Delta_{\bf i}(\varpi_3)$ for $\SL_6(\mathbb C)$ is not an integral polytope (see \cite{Ste19} for details).  
	
	As far as the authors know, the face structure, the integrality, the unimodular equivalences of string polytopes are not yet understood in general. 
	It contrasts with the Gelfand--Cetlin polytope, which is (unimodularly equivalent to) the string polytope with the standard word ${\bf i}_0$.
	The combinatorics of Gelfand--Cetlin polytopes are generalized and well-understood (see \cite{DLT, ABS, FFL, GKT, ACK} for instance).

	This paper aims to expand the understanding of string polytopes. 
	Specifically, we completely classify the unimodular equivalence class of the Gelfand--Cetlin polytope $\Delta_{{\bf i}_0}(\lambda)$ in the string polytopes as an attempt to answer the following question.

\begin{question}
Classify the unimodular equivalence classes of string polytopes. 
\end{question}

	In order to describe the criterion in terms of words, we introduce a function  and an operator on the commutation classes of reduced words: an {\em $\bullet$-(co)index} and a {\em $\bullet$-contraction} of ${\bf i}$ in Sections~\ref{ssecDAndAIndices} and~\ref{ssecGeneratingReducedWords}.
	The {\em index} measures
	how ${\bf i}$ is ``{\em far}'' from the canonical word and its involution
	\[
		{\bf i}_D := (1,2,1,\dots, n,n-1,\dots,1)\quad \text{or} \quad {\bf i}_A := (n,n-1,n,\dots, 1,2,\dots,n).   
	\]
	On the other hand, a {\em contraction} is an operator acting on each ${\bf i}$ so as to produce a new reduced word of the longest element in the Weyl group 
	of $\SL_{n}(\C)$ by ``{\em contracting}" the substring ${\bf i}_D$ or ${\bf i}_A$ in ${\bf i}$. 
	The contraction map is algorithmic. It produces a simple test on the reduced word for the unimodular equivalence. 
	
	Also, we verify that every string inequality and $\lambda$-inequality is non-redundant as long as $G$ is of Lie type $A$ and $\lambda$ is regular. 
	Namely, each inequality supports different facets, i.e. codimension-one faces.
	The number of facets of string polytopes can tell whether the string polytope is unimodularly equivalent to the Gelfand--Cetlin polytope or not. 
		
	Now we state our main theorem as follows.
			
	\begin{thmx}[Theorem~\ref{thm_GC_type_string_polytope}]\label{theorem_A}
		Let ${\bf i}$ be a reduced word of the longest element in the Weyl group of $\SL_{n+1}(\C)$.
		Let $\lambda$ be a regular dominant integral weight. 
		Then the following are equivalent. 
		\begin{enumerate}
			\item The string polytope $\Delta_{\bf i}(\lambda)$ is unimodularly equivalent to the Gelfand--Cetlin polytope $\Delta_{{\bf i}_0}(\lambda)$. 
			\item The string polytope $\Delta_{\bf i}(\lambda)$ has exactly $n(n+1)$ many facets.
			\item The associated string cone $C_{\bf i}$ is simplicial.
			\item There exists a sequence $(\sigma_1,\dots,\sigma_n) \in \{A, D\}^n$ such that 
\[
\mathrm{ind}_{\sigma_k}\left( C_{\sigma_{k+1}} \circ \cdots \circ C_{\sigma_n} ({\bf i}) \right) = 0  \quad \text{ for all }k = n,\dots,1.
\]
			Here $\mathrm{ind}_\bullet$ 
			denotes the $\bullet$-index of ${\bf i}$ and $C_\bullet$ denotes a $\bullet$-contraction where $\bullet = D$ or $A$ 
			\textup{(}see Section \ref{secBraidMovesAndIndices}\textup{)}.
		\end{enumerate}
	\end{thmx}	
	
	This paper is organized as follows. In Section \ref{secGleizerPostnikovDescription}, 
	we review Gleizer--Postnikov's description of string cones which fits into our purpose. In Section \ref{secBraidMovesAndIndices}, we define the notion of 
	an index (of various type) and also define contraction and extension operators acting on each reduced words. 
	The non-redundancy of string inequalities and $\lambda$-inequalities is discussed in Section \ref{secNonRedundancyOfStringInequalities}.  
	The proof of Theorem \ref{theorem_A} will be given in the rest of the sections.

\subsection*{Acknowledgements} 
The first author was supported by the National Research Foundation of Korea(NRF) grant funded by the Korea government(MSIP; Ministry of Science, ICT \& Future Planning) (NRF-2017R1C1B5018168). 
The third author was supported by IBS-R003-D1 and Basic Science Research Program through the National Research
Foundation of Korea (NRF) funded by the Ministry of Science and ICT (No. 2016R1A2B4010823).   
The fourth author was supported by IBS-R003-Y1.

\section{Gleizer--Postnikov description}
\label{secGleizerPostnikovDescription}
In this section, we recall the description of the string cone using wiring diagrams introduced by Gleizer--Postnikov~\cite{GP}.
Before doing that, we first introduce the notations.
	Consider the Lie algebra $\frak{sl}_{n+1}(\C)$ of $\textup{SL}_{n+1}(\C)$ with a fixed positive Weyl chamber ${\bf t}^*_+$. 
We fix an ordering on the simple roots $\alpha_1, \dots, \alpha_n$ as in Figure~\ref{figure_Dynkin}.
\begin{center}
	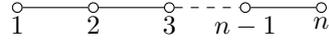
\begin{figure}[H]
		\begin{tikzpicture}[inner sep = 0.45mm]
		\node (1) at (0,0) [circle, draw] {};
		\node (2) at (1,0) [circle, draw] {};
		\node (22) at (2,0) [circle, draw] {};
		\node (3) at (3,0) [circle, draw] {};
		\node (4) at (4,0) [circle, draw] {};
		
		\node [below] at (1.south) {$1$};
		\node [below] at (2.south) {$2$};
		\node [below] at (22.south) {$3$};
		\node [below] at (3.south) {$n-1$};
		\node [below] at (4.south) {$n$};
		
		\path (1) edge (2) 
		(2) edge (22)
		(3) edge (4);
		\path[dashed] (22) edge (3);
		\end{tikzpicture}
		\caption{Dynkin diagram of Lie type $A_n$.}
		\label{figure_Dynkin}
	\end{figure}
\end{center}
\vs{-1.0cm}
Hence we have the enumeration on Cartan integers $(c_{ij})_{1 \leq i,j \leq n}$:
\begin{equation}\label{eq_Cartan_integers}
c_{ij} := \langle \alpha_i, \alpha_j^{\vee} \rangle = \begin{cases}
-1 & \text{ if } |i-j| = 1, \\
2 & \text{ if } i = j, \\
0 & \text{ otherwise}
\end{cases}
\end{equation}
where $\alpha_{i}^{\vee}$ is the coroot of $\alpha_i$. Also we let $\{\varpi_1,\dots,\varpi_n\}$ be the fundamental weights with respect to ${\bf t}^*_+$, which are characterized by the relation $\langle \varpi_i, \alpha_j^{\vee} \rangle = \delta_{ij}$. 
We call a weight $\lambda = \lambda_1 \varpi_1 + \cdots + \lambda_n \varpi_n$ 
 {\em regular} if $\lambda_i \neq 0$ for all~$i = 1,\dots,n$, and 
 {\em dominant} if $\lambda_i \geq 0$ for all $i = 1,\dots,n$. 

	Let $\mathfrak{S}_{n+1}$ be the symmetric group on $[n+1] := \{1,2,\dots,n+1\}$, which can be identified with the Weyl group of $\SL_{n+1}(\C)$. Set 
	$$N := \frac{n(n+1)}{2}.$$
	We note the dimension of the full flag variety $\textup{SL}_{n+1}(\C)/B$ of Lie type $A$ is $N$. Moreover, the length of the longest element in the symmetric group $\mathfrak{S}_{n+1}$ is also $N$.
	For the longest element $w_0$ in $\mathfrak{S}_{n+1}$, we denote the set of reduced words of the longest element $w_0$ in $\mathfrak{S}_{n+1}$ by 
	\[
		\Sigma_{n+1} := \{ {\bf i} = (i_1, \dots, i_N) \in [n]^N \mid w_0 = s_{i_1}s_{i_2} \cdots s_{i_N} \}, \quad \quad 	\]
	where $s_i = (i, i+1)$ is the simple transposition, which corresponds to the simple root $\alpha_i$ for $i=1,\dots,n$. 	
	For each reduced word ${\bf i} \in \Sigma_{n+1}$, 
	one can associate a convex rational polyhedral cone, called the {\em string cone} and denoted by $C_{\bf i}$, in $\R^N$ as in Littelmann \cite{Li} and Berenstein--Zelevinsky \cite{BZ} {\footnote{Their descriptions work for other Lie types as well. In this article, we will focus on string cones and polytopes of Lie type $A$.}}.	 
	Each lattice point in $C_{\bf i}$ parametrizes a dual canonical basis element of the quantized universal enveloping algebra associated to $\frak{sl}_{n+1}(\C)$.

	For each dominant integral weight $\lambda = \lambda_1 \varpi_1 + \cdots + \lambda_n \varpi_n $,
	Littelmann \cite{Li} constructed a convex polytope $\Delta_{\bf i}(\lambda)$, called the {\em string polytope}, where the set of lattice points therein is in one-to-one correspondence with 
	the dual canonical basis of the highest weight module (with highest weight $\lambda$) for $\SL_{n+1}(\C)$. 
	The string polytope $\Delta_{\bf i}(\lambda)$ is obtained as the intersection of the string cone $C_{\bf i}$ and so-called the {\em $\lambda$-cone} $C_{\bf i}^\lambda$. 
	We shall review Gleizer--Postnikov's (respectively, Rusinko's) description of the string cone (respectively, the $\lambda$-cone) of Lie type $A$ in the wiring diagram 
	(see \cite{GP, Ru}).

\subsection{Wiring diagrams and rigorous paths }
\label{ssecWiringDiagrams}\label{ssecRigorousPaths}

	Any reduced word ${\bf i} = (i_1, \dots, i_N) \in \Sigma_{n+1}$ can be represented by a {\em wiring diagram} (also called a {\em pseudoline arrangement}, or a {\em string diagram}). 
	The diagram corresponding to the reduced word ${\bf i}$ will be denoted by $G({\bf i})$.
	As depicted in Figure \ref{figure_wiring_diagram_GC}, 
	the wiring diagram $G({\bf i})$ consists of a family of~$(n+1)$ vertical piecewise straight lines such that 
	each pair of wires must intersect exactly once.  
	The lines are labeled by $\ell_1, \ell_2, \ldots , \ell_{n+1}$ and the upper end and lower end of each line $\ell_k$ are labeled by $U_k$ and $L_k$, respectively. 
	We call $\ell_k$ the $k$th \emph{wire}.

	The crossing patterns of pairs of wires are determined by ${\bf i}$.
	Specifically, the position of the $j$th crossing (from the top) should be located on the $i_j$th column of $G({\bf i})$ (see Figure \ref{figure_wiring_diagram_GC}). 
	We call each crossings {\em nodes} and name them as $t_1, t_2, \ldots, t_N$ from the top to the bottom. 
		
	\begin{figure}[h]
		\scalebox{0.92}{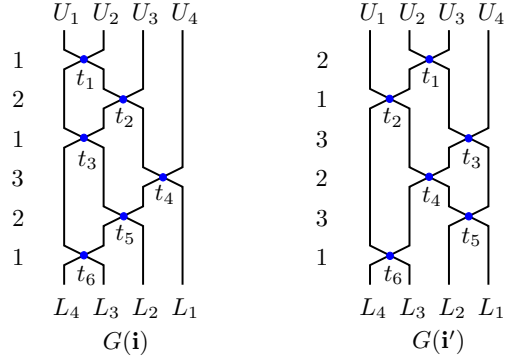} \vspace{-0.2cm}
		\caption{\label{figure_wiring_diagram_GC} Wiring diagrams for ${\bf i} = (1,2,1,3,2,1)$ and ${\bf i}' = (2,1,3,2,3,1)$.}	
	\end{figure}
	
	A {\em rigorous path} is an oriented path on $G({\bf i})$ defined as follows. 
	For each $k \in [n]$, let $G({\bf i}, k)$ be the wiring diagram $G({\bf i})$ together with the orientation on the wires where the first $k$ wires $\ell_1, \ldots, \ell_k$ are oriented 
	upward and the other wires $\ell_{k+1}, \ldots, \ell_{n+1}$ are oriented downward (see Figure~\ref{figure_wd_oriented}). 
	
	\begin{definition}[{\cite[Section 5.1]{GP}}]\label{definition_rigorous_path}
		For any $k \in [n]$, a {\em rigorous path} (or a {\em Gleizer--Postnikov path}) is an oriented path on $G({\bf i}, k)$ obeying the following properties:
		\begin{itemize}
			\item it starts at $L_k$ and ends at $L_{k+1}$, 
			\item it respects the orientation of $G({\bf i}, k)$,
			\item it passes through each node at most once, and
			\item it does \emph{not} include {\em forbidden fragments} given in Figure~\ref{figure_avoiding}.
		\end{itemize}
		We denote the set of rigorous paths by $\mathcal{GP}({\bf i})$. 

	\begin{figure}[h]
		\scalebox{0.92}{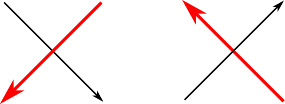}
			\vspace{-0.1cm}
		\caption{\label{figure_avoiding} Forbidden fragments.}
	\end{figure}
	\vspace{-0.2cm}
	\end{definition}
	
	\begin{remark}\label{remark_avoiding}
		A forbidden fragment in Figure~\ref{figure_avoiding} occurs when $\ell_i$ crosses over $\ell_j$ such that 
		\begin{itemize}
			\item $i > j$ where the orientation of both wires is downward, or 
			\item $i < j$ where the orientation of both wires is upward.  
		\end{itemize}	
	\end{remark}
		
	A node $t$ is called a {\em peak} of a rigorous path $P \in \mcal{GP}({\bf i})$ if $t$ is a local maximum of the path $P$ with respect to the height of the diagram $G({\bf i})$.
	Note that $P$ may have many peaks. As in Figure \ref{figure_wd_oriented}, 
	the path in the third diagram has a unique peak at $t_1$ while the path in the last diagram peaks at two nodes $t_2$ and~$t_3$. 
	
\vspace{0.1cm}	
		\begin{figure}[h]
		\scalebox{0.92}{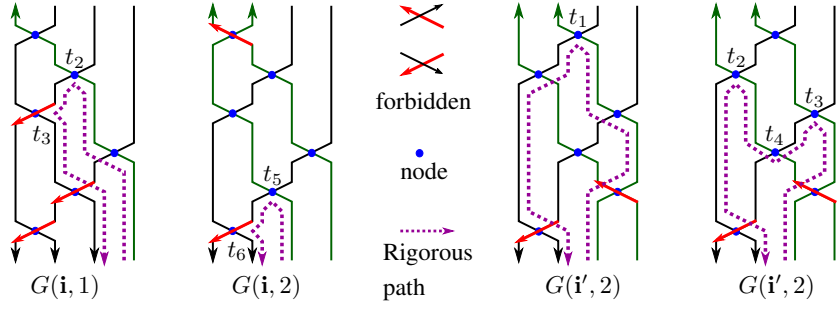}
		\vspace{-0.1cm}	
		\caption{\label{figure_wd_oriented} Oriented wiring diagrams for ${\bf i} = (1,2,1,3,2,1)$ and ${\bf i}' = (2,1,3,2,3,1)$.}
	\end{figure}	

	We will use two expressions for a rigorous path $P$: a node-expression and a wire-expression.
	Following the notation of \cite{GP}, we express a rigorous path $P$ in $G({\bf i}, k)$ by 
	\begin{equation}\label{equation_node_expression}
		P = (L_k \rightarrow t_{j_1} \rightarrow \cdots \rightarrow t_{j_s} \rightarrow L_{k+1}),
	\end{equation}
	where $t_{j_1},  \dots, t_{j_s}$ are the nodes at which the path $P$ crosses from one wire to another wire. 
	We will call \eqref{equation_node_expression} the {\em node-expression} of $P$ and  
	denote by $\node(P) := \{t_{j_1}, \dots, t_{j_s} \}$.
	For instance, the path in the first (resp. fourth) diagram in Figure~\ref{figure_wd_oriented} is expressed as $L_1 \to t_2 \to t_3 \to L_2$ (resp. $L_2 \to t_3 \to t_4 \to t_2 \to L_3$).
	
	Also, a rigorous path can be expressed by recording the wires in order traveling through. In other words, the rigorous path $P$ given in \eqref{equation_node_expression} can be written as
	\begin{equation}\label{equation_wire_expression}
		\ell_{r_1} \rightarrow \cdots \rightarrow \ell_{r_{s+1}} \quad \quad \quad r_1 = k, \quad r_{s+1} = k+1
	\end{equation}
	where the node $t_{j_t}$ is at the intersection of $\ell_{r_t}$ and $\ell_{r_{t+1}}$ for each $t=1,\dots,s$.
	The expression \eqref{equation_wire_expression} is called a {\em wire-expression}. 
	For instance, the path in the first (resp. fourth) diagram in Figure~\ref{figure_wd_oriented} is expressed as $\ell_1 \to \ell_3 \to \ell_2$ (resp. $\ell_2 \to \ell_4 \to \ell_1 \to \ell_3$).
	
 \vs{0.1cm}	 
\subsection{String inequalities}\label{ssecStringInequalities}

We now introduce defining inequalities of the string cone $C_{\bf i}$.
	
	\begin{definition}\label{definition_string_inequality}
		Let $P$ be a rigorous path in $G({\bf i}, k)$ for some $k \in [n]$. The {\em string inequality associated to $P$} is defined by 
		\begin{equation}\label{equation_stringcone}
			\sum a_i t_i \geq 0, \quad \quad \text{ where }a_i := \begin{cases}
				1 & \text{if $P$ travels from $\ell_r$ to $\ell_s$ at $t_i$ and $r < s$}, \\ 
				-1 & \text{if $P$ travels from $\ell_r$ to $\ell_s$ at $t_i$ and $r > s$}, \\ 
				 0 & \text{otherwise}.
			\end{cases}
		\end{equation}
	\end{definition}

	\begin{theorem}[{\cite[Corollary 5.8]{GP}}]\label{theorem_GP_BZ}
		The string cone $C_{\bf i}$ coincides with the set of points satisfying the string inequalities in Definition \ref{definition_string_inequality}.
	\end{theorem}

	\begin{example}\label{example_string_cone_GC}
		Let ${\bf i} = (1,2,1,3,2,1)$ and ${\bf i}' = (2,1,3,2,3,1)$. Then the corresponding wiring diagrams are given in Figure \ref{figure_wiring_diagram_GC}.
	\noindent 
	Then the string cones $C_{\bf i}$ and $C_{\bf i'}$ are respectively defined as follows:
	\[
		\begin{cases}
			\text{ $G({\bf i}, 1)$} \colon t_1 \geq 0, t_2 - t_3 \geq 0, t_4 - t_5 \geq 0, \\
			\text{ $G({\bf i}, 2)$}  \colon t_3 \geq 0, t_5 - t_6 \geq 0, \\
			\text{ $G({\bf i}, 3)$}  \colon t_6 \geq 0, 
		\end{cases}
		\, \text{and} \,\,
		\begin{cases}
			G({\bf i'},1) \colon t_5 \geq 0, \\
			G({\bf i'},2) \colon t_1 \geq 0, t_2 - t_5 \geq 0, t_3 - t_6 \geq 0, \\
                      \hskip 4em t_2 + t_3 - t_4 \geq 0, t_4 - t_5 - t_6 \geq 0, \\
			G({\bf i'},3) \colon t_6 \geq 0.
		\end{cases}
	\]
	\end{example}
\begin{remark}\label{rmk_number_of_string_cone_iequalities}
Notice that in Example \ref{example_string_cone_GC}, the number of inequalities for $C_{\bf i}$ is six, while that for $C_{\bf i'}$ is seven. 
Depending on a choice~${\bf i}$ of reduced words, the number of inequalities for string cone may vary.
\end{remark}

\vs{0.1cm}	 
\subsection{$\lambda$-inequalities}
\label{ssecLambdaInequalities}
	
	As explained in the introduction, for a given reduced word ${\bf i}$ and a  dominant weight $\lambda$, 
	a string polytope $\Delta_\lambda({\bf i})$ is obtained by intersecting two convex cones, where one is a string cone $C_{\bf i}$ and the other one is so-called a $\lambda$-cone 
	denoted by $C_{\bf i}^\lambda$. We explain how the $\lambda$-cone is defined.
	
	\begin{definition}[{\cite[Lemma 1]{Ru}}]\label{definition_lambda_inequality}
		Let ${\bf i} = (i_1, \dots, i_N) \in \Sigma_{n+1}$ and $\lambda = \lambda_1 \varpi_1 + \cdots + \lambda_n \varpi_n$ be a  dominant weight.
		For each node $t_j$ in $G({\bf i})$, the {\em $\lambda$-inequality associated to $t_j$} is defined by 
		\begin{equation}\label{equation_lambdacone}
			t_j \leq \lambda_{i_j} + \sum_{k > j} a_k t_k, \quad \quad \text{ where } a_k := \begin{cases}
				1 & \text{if the node $t_k$ is one column to the right or left of $t_j$}, \\
				-2 & \text{if the node $t_k$ is in  the same column as $t_j$}, \\
				0 & \text{otherwise}.
			\end{cases}
		\end{equation}
		The {\em $\lambda$-cone}, denoted by $C_{\bf i}^\lambda$, is the collection of points in $\R^N$ satisfying all $\lambda$-inequalities. 
	\end{definition}
	
	\begin{theorem}[{\cite{Li}}]\label{theorem_Li}
		For any ${\bf i} \in \Sigma_{n+1}$ and any 
dominant 
weight $\lambda$, we have 
		\[
			\Delta_{\bf i}(\lambda) = C_{\bf i} \cap C_{\bf i}^\lambda.
		\]
	\end{theorem}
	
	\begin{example}\label{example_lambda_inequality}
		Consider ${\bf i} = (1,2,1,3,2,1)$ as in Example \ref{example_string_cone_GC} and a dominant weight $\lambda = \lambda_1 \varpi_1 + \lambda_2 \varpi_2 + \lambda_3 \varpi_3$. 
		From $G({\bf i})$ in Figure~\ref{figure_wiring_diagram_GC}, the inequalities for $C_{\bf i}^\lambda$ can be written as follows.
		\[
			\begin{array}{lllll}
				t_1 \leq \lambda_1 + t_2 - 2t_3 + t_5 - 2t_6; & & t_2 \leq \lambda_2 +t_3 +t_4 -2 t_5+t_6; & & t_3 \leq \lambda_1 + t_5 -2t_6;\\
				t_4 \leq \lambda_3 +t_5; & & t_5 \leq \lambda_2 +t_6; & & t_6 \leq \lambda_1.\\
			\end{array}
		\]
	\end{example}

\vspace{0.09cm}
\section{Braid moves and indices}
\label{secBraidMovesAndIndices}

	As we have seen in Example~\ref{example_string_cone_GC} and Remark~\ref{rmk_number_of_string_cone_iequalities}, the combinatorics of the string cone $C_{\bf i}$, especially the number of supporting hyperplanes of $C_{\bf i}$, depends on the choice of 
	a reduced word ${\bf i} \in \Sigma_{n+1}$. 
	In this section, we introduce the notion of {\em an index} of ${\bf i}$
	which will be used in Section \ref{secSimplicialStringCones} to classify all reduced words whose string cones are {\em simplicial}, that is, having the minimal number of facets.
	To begin with, we establish some notations: for any word ${\bf i} = (i_1, \dots, i_k)$
and $a \in \Z_{\geq 0}$, 
	\begin{itemize}
		\item ${\bf i} \geq a$ (respectively ${\bf i} \leq a$) if $i_j \geq a$ (respectively $i_j \leq a$) for every $j$.
		\item ${\bf i} + a := (i_1 + a, \dots, i_k + a)$. 
	\end{itemize}
	
\subsection{Braid moves}
\label{ssecBraidMoves}
	According to Tits' Theorem \cite{Ti}, every pair of reduced words in $\Sigma_{n+1}$ is connected by a sequence of the following moves (called {\em braid moves}): 
	\begin{itemize}
		\item (2-move) exchanging $(i,j)$ with $(j,i)$ for $|i-j| > 1$, i.e., $s_i s_j = s_j s_i$.
		\item (3-move) exchanging $(i,j,i)$ with $(j,i,j)$ for $|i-j| = 1$, i.e., $s_i s_{i+1} s_i  = s_{i+1} s_i s_{i+1}$.
	\end{itemize}
	The braid moves are also described using wiring diagrams as follows.
	
	\begin{figure}[h]
		\scalebox{0.95}{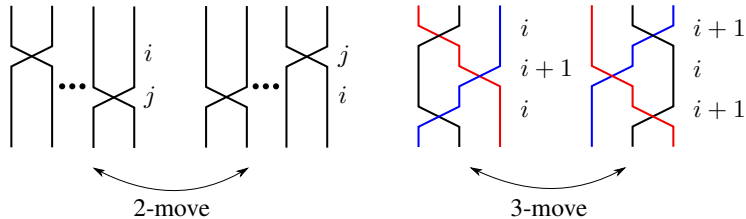}
		\vspace{-0.2cm}
		\caption{\label{figure_braid_moves} Braid moves.}
	\end{figure}
			\vspace{-0.1cm}

\subsection{$D$- and $A$-indices}
\label{ssecDAndAIndices}

	From the definition, it immediately follows that if two reduced words ${\bf i}$ and ${\bf i'}$ differ by 2-move $s_is_j \leftrightarrow s_js_i$, then the corresponding string cones
	$C_{\bf i}$ and $C_{\bf i'}$ also differ by the change of coordinates $t_i \leftrightarrow t_j$. 
	
	\begin{lemma}\label{lemma_2_move}
		If ${\bf i} \in \Sigma_{n+1}$ is obtained from ${\bf i'} \in \Sigma_{n+1}$ by a sequence of 2-moves, then $C_{\bf i}$ and $C_{\bf i'}$ are unimodularly equivalent. 
		Moreover, string polytopes $\Delta_{\mathbf i}(\lambda)$ and $\Delta_{\mathbf i'}(\lambda)$ are unimodularly equivalent for any  dominant weight $\lambda$.
	\end{lemma}
	
	Now, define an equivalence relation on $\Sigma_{n+1}$ such that 
	\begin{equation}\label{eq_def_of_2_move}
		{\bf i} \sim {\bf i'} \quad \Leftrightarrow \quad \text{${\bf i}$ and ${\bf i'}$ are related by a sequence of 2-moves}
	\end{equation}
	and denote the set of equivalence classes\footnote{Those equivalence classes are called the {\em commutation classes}. See \cite{B} for example.}
	by $\widetilde{\Sigma}_{n+1} := \Sigma_{n+1} / \sim$.  We first observe the following:
	
	\begin{proposition}\label{proposition_D}
		For any ${\bf i} = (i_1, \dots, i_N) \in \Sigma_{n+1}$, we can find a sequence of 2-moves such that 
		\[
			{\bf i} \sim (\underbrace{i_1', \dots, i_u'}_{=: ~{\bf i}_D^-}, \underbrace{n, n-1, \dots, 2, 1}_{=: ~D_n}, \underbrace{i_{u+n+1}', \dots, i_N'}_{=: ~{\bf i}^+_D}) 		
		\]
		for some integer $u\geq 0$ and $i_j' \in [n]$. Similarly, there exist an integer $v \geq 0$ and $i_j''$'s in $[n]$ such that 
		\[
			{\bf i} \sim (\underbrace{i_1'', \dots, i_v''}_{=: ~{\bf i}_A^-}, \underbrace{1, 2, \dots, n-1, n}_{=: ~A_n}, \underbrace{i_{v+n+1}'', \dots, i_N''}_{=: ~{\bf i}^+_A})
		\]
		where `$D$' and `$A$' stand for `descending' and 'ascending', respectively.
	\end{proposition}
	
	\begin{proof}
		We provide an algorithmic proof pictorially as in Figure \ref{figure_proof}.
		Note that the $(n+1)$th wire $\ell_{n+1}$ crosses each wire exactly once so that there are $n$ nodes lying on $\ell_{n+1}$.
		We denote by $t_{j_1}, \dots, t_{j_n}$ the nodes on $\ell_{n+1}$ in order such that 
		\[
			i_{j_1} = 1, i_{j_2} = 2, \dots, i_{j_n} = n, \quad \quad j_1 > j_2 > \cdots > j_n.
		\]
		\vspace{-0.1cm}
		\begin{figure}[h]
			\scalebox{0.95}{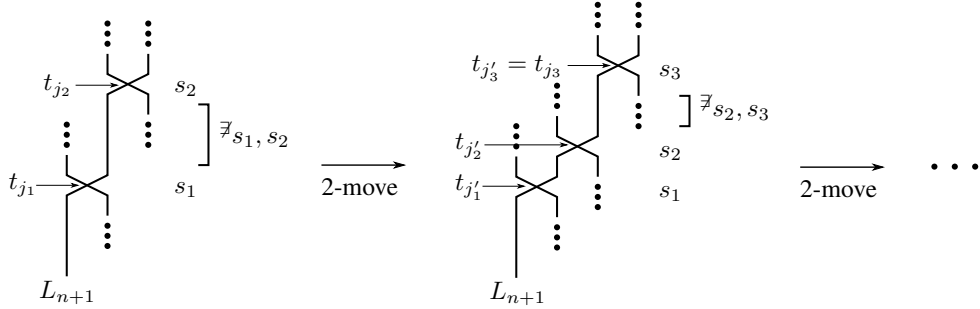}
			\caption{\label{figure_proof} Proof of Proposition \ref{proposition_D}.}
		\end{figure}	
		
		Neither $1$ nor $2$ can appear in the subword $(i_{j_2}, i_{j_2 + 1}, \dots, i_{j_2 + (j_1 - j_2)} = i_{j_1})$ of ${\bf i}$ except for $i_{j_1}$ and $i_{j_2}$ because 
		\begin{itemize}
			\item there is no other node on $\ell_{n+1}$ between $t_{j_1}$ and $t_{j_2}$, and 
			\item any node equal to 1 or 2 between $t_{j_1}$ and $t_{j_2}$ should be on $\ell_{n+1}$.
		\end{itemize}
		In other words, setting $B_1 = (i_{j_2 + 1}, \dots, i_{j_1 - 1})$, we see
		\[
			(i_1, \dots, \underbrace{i_{j_2}, \dots , i_{j_1}}_{=~(2,B_1,1)}, \dots, i_N), \quad \quad 1, 2 \not \in B_1.
		\]
		Applying 2-moves interchanging $B_1$ and $i_{j_1} (=1)$ successively, we get 
		\[
			(i_1, \dots, i_{j_2}, B_1, i_{j_1}, \dots, i_N) \quad \sim \quad 
			(i_1, \dots, i_{j_2}, i_{j_1}, B_1, \dots, i_N) =: {\bf i}_1 .
		\]
		Letting ${\bf i}_1 = (i_1', i_2',\dots, i_N')$, the new positions of $i_{j_k}$'s will be denoted by $(j_1', j_2' ,\dots, j_n')$. Namely,
		\[
			j_1' = j_1 - |B_1| = j_2 + 1, \quad j_k' = j_k \hs{0.1cm} \text{for $k \geq 2$.} 
		\]
				
		We then apply the same argument to ${\bf i}_1$. We see that there is neither 2 nor 3 in the subword $(i'_{j_3'}, i'_{j_3' + 1}, \dots,  i'_{j_2'}, i'_{j_1'})$ of ${\bf i}_1$ except for 
		$i'_{j_3'}$ and $i'_{j_2'}$, where $j_2' + 1 = j_1'$. See Figure \ref{figure_proof}. Thus we have
		\[
			(i'_1, \dots, \underbrace{i'_{j_3'}, \dots , i'_{j_2'}, i'_{j_1'}}_{=~(3,B_2,2,1)}, \dots, i'_N), \quad \quad 2, 3 \not \in B_2.
		\]
		Since every component of $B_2$ is either $1$ or greater than $3$, we may rearrange $B_2$ using 2-moves so that
		\[
			B_2 \quad \sim \quad B_2' = (B_2^-, B_2^+) \quad \text{where ~$B_2^- \leq 1$ ~and ~$B_2^+ \geq 4$.}
		\]
		Consequently, we obtain 
		\[
		\begin{split}
			(i'_1, \dots, i'_{j_3'}, B_2 , i'_{j_2'}, i'_{j_1'}, \dots, i'_N) \hs{0.1cm} &\sim \hs{0.1cm}
			(i'_1, \dots, i'_{j_3'}, B_2^-, B_2^+ , i'_{j_2'}, i'_{j_1'}, \dots, i'_N)\\
			 \hs{0.1cm} &\sim \hs{0.1cm} 
			(i'_1, \dots, B_2^-, \underbrace{i'_{j_3'}, i'_{j_2'}, i'_{j_1'}}_{=~(3,2,1)}, B_2^+ , \dots, i'_N).
		\end{split}
		\]
		We apply this procedure inductively and complete the proof for the first statement of the proposition.
		
		The second statement is also similarly proved, where the difference with the previous argument
		is that we need to use the first wire $\ell_1$ with the exactly $n$ nodes on it, instead of $\ell_{n+1}$. 
	\end{proof}
	
	\begin{example}
		Consider the reduced word 
		\[
			{\bf i} = (1,2,3,4,5,1,2,3,4,1,2,3,1,2,1) \in \Sigma_6. 
		\]
		Then, we have 
		\[
			\begin{array}{lll}
				{\bf i} = (1,2,3,4,5,1,2,3,4,1,2,3,1, \underline{2,1}) & \sim & (1,2,3,4,5,1,2,3,4,1,2,3,1,\underline{2,1}) := {\bf i}_1, \\
				{\bf i}_1 = (1,2,3,4,5,1,2,3,4,1,2,\underline{3,1,2,1}) & \sim & (1,2,3,4,5,1,2,3,4,1,2,\underline{1,3,2,1}) := {\bf i}_2, \\
				{\bf i}_2 = (1,2,3,4,5,1,2,3,\underline{4,1,2,1,3,2,1}) & \sim & (1,2,3,4,5,1,2,3,\underline{1,2,1,4,3,2,1}) := {\bf i}_3, \\
				{\bf i}_3 = (1,2,3,4,\underline{5,1,2,3,1,2,1,4,3,2,1}) & \sim & (1,2,3,4,\underline{1,2,3,1,2,1,5,4,3,2,1}) := {\bf i}_4. \\
			\end{array} 
		\]
		Thus we obtain ${\bf i} \sim (\underbrace{1,2,3,4,1,2,3,1,2,1}_{{\bf i}_D^-},\underbrace{5,4,3,2,1}_{D_5})$ and ${\bf i}_D^+ = \emptyset$.
	\end{example}
	
	\begin{definition}\label{definition_index}
		For a given ${\bf i} \in \Sigma_{n+1}$, by Proposition \ref{proposition_D}, we have 
		\[
			{\bf i} \quad \sim \quad {\bf i}_D^- ~D_n ~{\bf i}_D^+ \quad \sim \quad {\bf i}_A^- ~A_n ~{\bf i}_A^+.
		\]
		We denote by 
		\[
			\begin{array}{lll}
				\mathrm{ind}_D({\bf i}) := |{\bf i}_D^+|, & & \mathrm{ind}_A({\bf i}) := |{\bf i}_A^+|, \\
				\mathrm{coind}_D({\bf i}) := |{\bf i}_D^-|, & & \mathrm{coind}_A({\bf i}) := |{\bf i}_A^-|, \\
			\end{array}
			\quad \quad \text{where $|w| := \text{the length of $w$}$}
		\]
		and call them a {\em $D$-(co)index} and an {\em $A$-(co)index}, respectively. 
	\end{definition} 
	
	\begin{proposition}\label{proposition_well_defined}
		The notions $\mathrm{ind}_D({\bf i})$, $\mathrm{ind}_A({\bf i})$, $\mathrm{coind}_D({\bf i})$, and $\mathrm{coind}_A({\bf i})$ are well-defined.
	\end{proposition}
	
	\begin{proof}
		From the algorithm in the proof of Proposition \ref{proposition_D}, we can easily see that 
		\begin{itemize}
			\item $\mathrm{ind}_D({\bf i})$ = the number of nodes in $G({\bf i})$ below $\ell_{n+1}$, 
			\item $\mathrm{ind}_A({\bf i})$ = the number of nodes in $G({\bf i})$ below $\ell_{1}$, 
			\item $\mathrm{coind}_D({\bf i})$ = the number of nodes in $G({\bf i})$ above $\ell_{n+1}$,  
			\item $\mathrm{coind}_A({\bf i})$ = the number of nodes in $G({\bf i})$ above $\ell_{1}$.
		\end{itemize}	
		Also, since 2-moves do not change the combinatorial type of the wiring diagram $G({\bf i})$, i.e., 2-moves do not change the ``relative positions'' of a node and a wire, 
		the (co-)indices are well-defined up to 2-moves.
		This completes the proof.
	\end{proof}

\subsection{Generating reduced words}
\label{ssecGeneratingReducedWords}

	From Proposition \ref{proposition_D}, we can define the following two types of operations, called the {\em contraction} and the {\em extension}. 
	\begin{definition}\label{definition_contraction_extension} 
		For any ${\bf i} \in \Sigma_{n+1}$, let ${\bf i}_D^-$ and $~ {\bf i}_D^+$ (respectively, ${\bf i}_A^-$ and $~ {\bf i}_A^+$) 
		be given in Proposition \ref{proposition_D} such that 
		${\bf i}_D^-  ~D_n ~ {\bf i}_D^+$ (respectively, ${\bf i}_A^-  ~A_n ~ {\bf i}_A^+$) 
		is obtained from ${\bf i}$ using 2-moves as least as possible. Then the {\em $D$-contraction} (respectively, {\em $A$-contraction}) is defined by
	\[
		\begin{array}{ccccl}
			C_D & \colon & \Sigma_{n+1} & \rightarrow & \Sigma_n \\ 
				&   &   {\bf i} & \mapsto & {\bf i}_D^- \left({\bf i}_D^+ - 1\right)
		\end{array}
		\quad \left(
		\begin{array}{ccccl}
			C_A & \colon & \Sigma_{n+1} & \rightarrow & \Sigma_n \\ 
				&   &   {\bf i} & \mapsto & ({\bf i}_A^- -1) {\bf i}_A^+.
		\end{array}
		\right)		
	\]
	Conversely, for $0 \leq s \leq N$, a {\em $D$-extension at $s$} is defined by 
	\[
		\begin{array}{ccccl}
			E_D(s) & \colon & \Sigma_{n+1} & \rightarrow & \Sigma_{n+2} \\ 
				&   &   {\bf i} = (\underbrace{i_1, \dots, i_{N-s}}_{=: ~{\bf i}^-(s)}, \underbrace{i_{N-s+1}, \dots, i_N}_{=: ~ {\bf i}^+(s)}) 
				& \mapsto & {\bf i}^-(s) ~D_{n+1} ~ ({\bf i}^+(s)+1).
		\end{array}		
	\]
	Similarly, we define an {\em $A$-extension at $s$} such that
	\[
		\begin{array}{ccccl}
			E_A(s) & \colon & \Sigma_{n+1} & \rightarrow & \Sigma_{n+2} \\ 
				&   &   {\bf i} = (\underbrace{i_1, \dots, i_{N-s}}_{=: ~{\bf i}^-(s)}, \underbrace{i_{N-s+1}, \dots, i_N}_{=: ~ {\bf i}^+(s)}) 
				& \mapsto & ({\bf i}^-(s) +1) ~A_{n+1} ~ {\bf i}^+(s).
		\end{array}		
	\]
	\end{definition}	

	\begin{example}\label{example_contraction_extension} 
		We illustrate some examples of an extension as follows. 
		\begin{figure}[h]
			\scalebox{1}{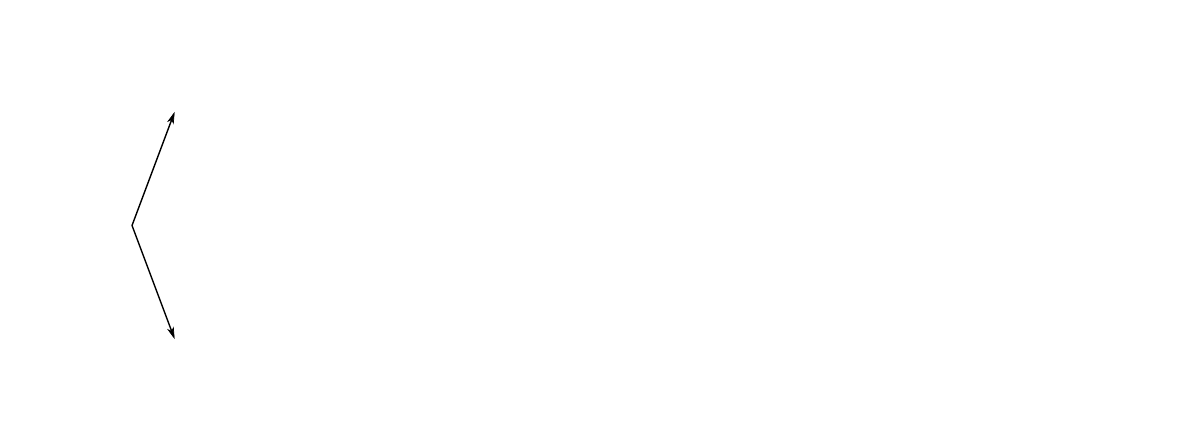}
		\end{figure}
		
	\noindent	
	For an example of a contraction, let us consider the following.
	\[
		\begin{array}{rll}
			{\bf i} = (2,1,4,3,5,4,2,1,3,\underline{2,5,4,3,5,1}) & \rightarrow & (2,1,4,3,5,4,2,1,3,\underline{2,1,5,4,3,5}),  \\
				  (2,1,4,3,5,4,2,1,\underline{3,2,1},5,4,3,5) & \rightarrow & (2,1,4,3,5,4,2,1,\underline{3,2,1},5,4,3,5), \quad \text{nothing changed}  \\						
				  (2,1,4,3,5,\underline{4,2,1,3,2,1},5,4,3,5) & \rightarrow & (2,1,4,3,5,\underline{2,1,4,3,2,1},5,4,3,5),  \\						
				  (2,1,4,3,\underline{5,2,1,4,3,2,1},5,4,3,5) & \rightarrow & (2,1,4,3,\underline{2,1,5,4,3,2,1},5,4,3,5),  \\				
		\end{array}
	\]
	Then ${\bf i}_D^- = (2,1,4,3,2,1)$ and ${\bf i}_D^+ = (5,4,3,5)$ and hence we get 
	\[
		C_D({\bf i}) = {\bf i}_D^- ({\bf i}_D^+ -1) = (2,1,4,3,2,1,4,3,2,4). 
	\]
	Furthermore, we can 
 check that $C_D^2({\bf i}) = C_D(C_D({\bf i})) = (2,1,3,2,1,3)$, $C_D^3({\bf i}) = (2,1,2)$, and $C_D^4({\bf i}) = (1)$.
	\end{example}

	\begin{remark}\label{remark_well_define}
		To explain why an extension or a contraction sends a reduced word to a reduced word, it is rather straightforward to explain using a wiring diagram as follows. 
		For ${\bf i} \in \Sigma_{n+1}$, we may restrict the wiring diagram $G(\bf {i})$ to $n$ wires $\ell_1, \dots, \ell_n$, removing $\ell_{n+1}$.
		The restriction also becomes a wiring diagram (for some reduced word in $\Sigma_n$) since each pair of wires still intersects at exactly one node.
		Roughly speaking, a $D$- or $A$-contraction corresponds to `removing $\ell_{n+1}$ or $\ell_1$', respectively, 
		and a $D$- or $A$-extension corresponds to `inserting a new wire $\ell_{n+2}$ or $\ell_1$', respectively.
	\end{remark}

	Finally, we give a remark on commutation classes in $\widetilde{\Sigma}_{n+1}$. 
	Using our notations, for $\bullet = A$ or $D$ we obtain the following sequence of maps
	\[
		\begin{array}{ccccl}\vs{0.1cm}
			\Sigma_{n+1} \times [0, N]& \stackrel{E_\bullet} \longrightarrow & \Sigma_{n+2} & \stackrel{\pi} \longrightarrow & \widetilde{\Sigma}_{n+2} \\ \vs{0.1cm}
			({\bf i}, s) & \mapsto & E_\bullet(s)({\bf i}) & \mapsto & \left[ E_\bullet(s)({\bf i}) \right]
		\end{array}
	\]
	where $\pi$ is the quotient map. Then Proposition \ref{proposition_D} implies that the composition $\pi \circ E_\bullet$ is surjective for $\bullet = A$ or $D$. 
	We can also easily check that the map 
	\[
		\begin{array}{ccl}\vs{0.1cm}
			\Sigma_{n+2} & \stackrel{(C_\bullet, \mathrm{ind}_\bullet)} \longrightarrow & \Sigma_{n+1} \times [0, N] \\ \vs{0.1cm}
			{\bf i} & \mapsto & (C_\bullet({\bf i}), \mathrm{ind}_\bullet({\bf i})) 
		\end{array}
	\]
	is surjective and $E_\bullet$ is a section of the map $(C_\bullet, \mathrm{ind}_\bullet)$, i.e., 
	\[
		(C_\bullet, \mathrm{ind}_\bullet) \circ E_\bullet = \mathrm{id}_{\Sigma_{n+1} \times [0, N]}
	\]
	for $\bullet = A$ or $D$. On the other hand, the map $E_\bullet$ is not surjective in general, however, one can show that 
	\begin{equation}\label{equation_commutation}
		E_\bullet  \circ (C_\bullet, \mathrm{ind}_\bullet) ({\bf i}) \sim {\bf i}
	\end{equation}
	for $\bullet = A$ or $D$.

\vspace{0.1cm}
\section{Non-redundancy of string inequalities}
\label{secNonRedundancyOfStringInequalities}

The aim of this section is to prove that every string inequality and $\lambda$-inequality for the string polytope associated to a regular dominant weight is non-redundant. 
For each reduced word  $\bf i$ of the longest element $w_0$ and each choice of regular dominant weight $\lambda$, we shall show that there is a one-to-one correspondence between the facets of $\Delta_{\bf i}(\lambda)$ and the inequalities in~\eqref{equation_stringcone} and~\eqref{equation_lambdacone}.
The main ingredient for the proof is the \emph{chamber variables} on the wiring diagram $G({\bf i})$ associated to ${\bf i}$, which will be introduced in Section~\ref{ssecChangeOfCoordinatesChamberVariables}.

\subsection{Chamber variables}
\label{ssecChangeOfCoordinatesChamberVariables}
	
		The wiring diagram $G({\bf i})$ divides the region into several (bounded or unbounded) {\em chambers}. 
		For each node $t_j$ in $G({\bf i})$, we denote by ${\scr C}_j$ the {\em chamber} having $t_j$ as a peak as depicted in Figure~\ref{figure_chamber_variable}. 
		For each chamber ${\scr C}_j$, let $I_j$ be the set of nodes contained in ${\scr C}_j$. Then $I_j$ is decomposed 
		into the disjoint union $I_j^+ \cup I_j^-$ where 
		\begin{itemize}
			\item $I_j^+$ consists of nodes in ${\scr C}_j$ in the same column as $t_j$,
			\item $I_j^-$ consists of nodes in ${\scr C}_j$ one column to the right or left of  $t_j$.
		\end{itemize}
		
		\begin{definition}\label{definition_chamber_variable}
			For each node $t_j$ in $G({\bf i})$, we define
			\begin{equation}\label{equation_chambervar}
				u_j := \sum_{t_k \in I_j^+} t_k - \sum_{t_k \in I_j^-} t_k
			\end{equation}
			and call it {\em the $i$th chamber variable.}
		\end{definition}		

		\vspace{-0.1cm}
		\begin{figure}[h]
			\scalebox{0.9}{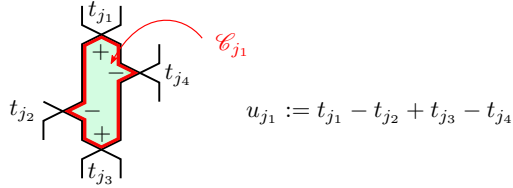}
			\vspace{-0.2cm}
			\caption{\label{figure_chamber_variable} Change of coordinates: chamber variables.}
		\end{figure}
			\vspace{-0.1cm}

		Under the change of coordinates $(t_1, \dots, t_N) \rightarrow (u_1, \dots, u_N)$ in~\eqref{equation_chambervar}, every string inequality associated to a rigorous path $P$
		can be expressed as 
		\[
		\sum_{\scr{C}_j \text{ is in a chamber enclosed by $P$}} u_j \geq 0.
		\]
		Also, each $\lambda$-inequality is written as 
		\begin{equation}\label{equation_change_coordinate}
			t_j \leq \lambda_{i_j} + \sum_{k > j} a_k t_k \quad \quad \Longleftrightarrow \quad \quad \sum_{\substack{k \geq j, \, i_{k} = i_j}} u_k \leq \lambda_{i_j}
		\end{equation}
		for $j = 1,\dots,N$. Here, we note that $i_k = i_j$ if and only if the nodes $t_{k}$ and $t_j$ are in the same column of $G({\bf i})$.

\begin{remark}		
		One notable feature of the expression of string inequalities and $\lambda$-inequalities in terms of chamber variables is that every coefficient of any chamber variable is either \emph{one} or \emph{zero}. This feature will be crucial for the proof of non-redundancy of the inequalities in~\eqref{equation_stringcone} and~\eqref{equation_lambdacone}.
\end{remark}
		
		\begin{example}
			Let us revisit Example~\ref{example_string_cone_GC}. As an example, we compare four inequalities in terms of node variables and chamber variables.
			\[
			\begin{array}{lllll}
				\mbox{\,} &\, &\mbox{Node variables} & \,  & \mbox{Chamber variables} \\
				\mbox{Figure}~\ref{figure_nodechamber_a}. &\, & t_1 \geq 0 & \Longleftrightarrow & u_1 + u_2 + u_3 + u_4 \geq 0 \\								
				\mbox{Figure}~\ref{figure_nodechamber_b}. &\, & t_2 + t_3 - t_4 \geq  0 & \Longleftrightarrow &  u_2 + u_3 + u_4 \geq 0 \\
				\mbox{Figure}~\ref{figure_nodechamber_c}. &\, & t_1 \leq \lambda_2 + t_2 + t_3 - 2 t_4 + t_5 + t_6 & \Longleftrightarrow & u_1 + u_4 \leq \lambda_2 \\
				\mbox{Figure}~\ref{figure_nodechamber_d}. &\, & t_3 \leq \lambda_3 + t_4 - 2 t_5 & \Longleftrightarrow & u_3 + u_5 \leq \lambda_3 \\								
			\end{array}
		\] 	
			\begin{figure}[h]
				\centering
				\begin{subfigure}[b]{0.1\textwidth}
	\scalebox{0.92}{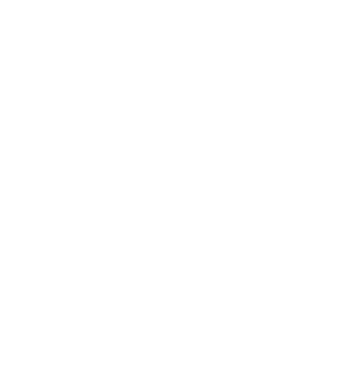}
\end{subfigure}				
				\begin{subfigure}[b]{0.2\textwidth}
{\makebox[0.55cm][r]{\makebox[0pt][l]{	\scalebox{0.92}{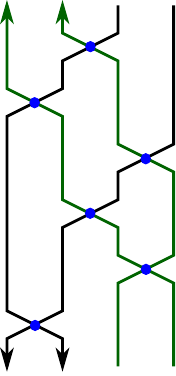}}}}
\caption{ \label{figure_nodechamber_a}}
\end{subfigure}
				\begin{subfigure}[b]{0.2\textwidth}
					\centering
	\scalebox{0.92}{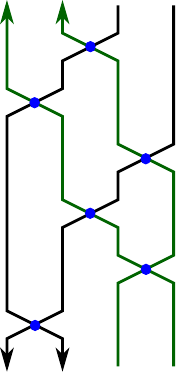}
\caption{ \label{figure_nodechamber_b}}
\end{subfigure}
				\begin{subfigure}[b]{0.2\textwidth}
					\centering
	\scalebox{0.92}{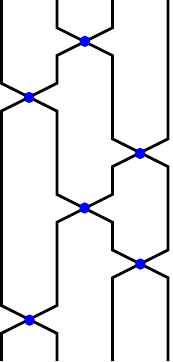}
	\caption{ \label{figure_nodechamber_c}}
\end{subfigure}
				\begin{subfigure}[b]{0.2\textwidth}
					\centering
	\scalebox{0.92}{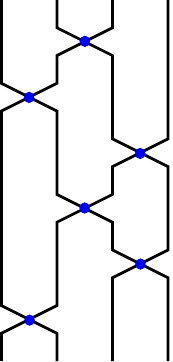}
\caption{ \label{figure_nodechamber_d}}
\end{subfigure}
		\caption{\label{figure_nodechamber} Node variables and chamber variables.}
	\end{figure}
	\vspace{-0.2cm}
		\end{example}

\subsection{Non-redundancy of the string and $\lambda$-inqualities}\label{subsecNonRedundancyOfStringInequalities}
Let $Q \subset \R^N$ be a full-dimensional polyhedron defined by the system of inequalities 
\begin{equation}\label{equation_phex}
\{ \rho_j ({\bf u}) \geq 0 \mid j = 1, \dots, \nu \}
\end{equation}
where $\mathbf{u} = (u_1,\dots,u_N)$ is a coordinate for $\R^N$.
We say that the inequality $\rho_k(\mathbf{u}) \geq 0$ is \emph{redundant} in $Q$ if it can be removed from the system. If every inequality in the system is not redundant, then 
we say that the system is {\em non-redundant.}
Equivalently, the system \eqref{equation_phex} is non-redundant if and only if the number of facet of $Q$ equals $\nu$.

We first show that the string inequalities given in \eqref{equation_stringcone} are not redundant.
Let us fix any reduced expression $\bf i$ of the longest element $w_0 \in \frak{S}_{n+1}$. 
Since the face structure of $C_{\bf i}$ is preserved under the change of variables in Definition~\ref{definition_chamber_variable}, 
it is enough to prove that string inequalities expressed in terms of  chamber variables in~\eqref{equation_chambervar} are non-redundant.

We assume that the string inequalities in \eqref{equation_stringcone} are given by 
\[ 
	\rho_1({\bf u}) \geq 0, \, \rho_2 ({\bf u}) \geq 0, \dots,  \rho_\nu ({\bf u}) \geq 0.
\]
where ${\bf u}$ denotes the coordinates of chamber variables introduced in \eqref{equation_chambervar}.  
Let ${\bf v}_j$ be the primitive inward normal vector to the hyperplane $\rho_j ({\bf u}) = 0$ in $C_{\bf i}$, i.e., $\rho_{j}(\mathbf{u}) = \langle {\bf v}_j, \bf{u} \rangle$.
We shall prove that each inequality $\rho_j({\bf u}) \geq 0$ is non-redundant in $C_{\bf i}$, which is equivalently saying that the dual cone of $C_{\bf i}$ has exactly $\nu$ rays.

\begin{lemma}\label{lemma_fnncc}
Any facet normal vector of $C_{\bf i}$ is \emph{not} contained in the cone generated by the others. 
\end{lemma}

\begin{proof}
Without any loss of generality, it is enough to prove that ${\bf v}_1$ is not in the cone generated by ${\bf v}_2, \cdots, {\bf v}_\nu$. 
Assume on the contrary that there exist $a_2,\dots,a_{\nu} \in \R_{\geq 0}$ such that
\begin{equation}\label{equ_vicom}
{\bf v}_1 =a_2  {\bf v}_2 + \cdots + a_\nu  {\bf v}_\nu
\end{equation}
for some $a_2,\dots,a_{\nu} \in \R_{\geq 0}$. 

Suppose that $\rho_1$ is obtained from a rigorous path $P_1$ starting from the lower end $L_{n+1-k}$ to the lower end $L_{n+2-k}$. 
For each $i = 1, \dots, n$, we denote by $\scr{U}_{i}$ a (unique) unbounded chamber containing both $L_{n+1-i}$ and $L_{n+2-i}$ and 
let $u_{(i)}$ be the chamber variable corresponding to $\scr{U}_{i}$.  
(For instance, in Figure~\ref{figure_nodechamber}, there are three chamber variables corresponding to unbounded chambers$\colon u_{(1)} = u_6, u_{(2)} = u_4,$ and $u_{(3)} = u_5$.) 

Notice that 
\begin{equation}\label{equation_coef}
\mbox{ 
(the coefficient of $u_{(i)}$ in ${\rho}_1$) 
$= 
\begin{cases} 
1 &\mbox{if $i = k$}
\\ 0 &\mbox{if $i \neq k$.} 
\end{cases}
$} 
\end{equation}
In other words, the entry of $\mathbf{v}_1$ corresponding to $u_{(i)}$ is $1$ if $i = k$ and $0$ otherwise.

Let $\{r_1, \dots, r_u\} \subset [2,\nu]$ be the set of indices such that $a_{r_i} \neq 0$. 
By comparing the coefficients of $u_{(k)}$ of the left hand side and the right hand side in~\eqref{equ_vicom}, we obtain 
\begin{equation}\label{equation_sumofcoeff}
1 = a_{r_1} + a_{r_2} + \cdots + a_{r_u}.
\end{equation}
Observe that each $\rho_{r_i}$'s are obtained from rigorous paths starting from the lower end $L_{n+1-k}$ to the lower end $L_{n+2-k}$. 
(Otherwise, there exists some $r_i$ such that $\rho_{r_i}$ possesses a non-zero unbounded chamber variable term for some $u_{(j)}$ with $j \neq k$ and this implies that $\rho_1$ also contains 
the non-zero term for $u_{(j)}$ which leads to a contradiction to \eqref{equation_coef}.) 

Let $P_{r_1}$ be the rigorous path corresponding to $\rho_{r_1}$. Then there exists a chamber $\scr{C}_{r_1}$ with the chamber variable, say $\widetilde{u}$, such that 
\begin{enumerate}
	\item $\scr{C}_{r_1}$ is contained in the region bounded by $P_{r_1}$ but not in one bounded by $P_1$, or 
	\item $\scr{C}_{r_1}$ is contained in the region bounded by $P_1$ but not in one bounded by $P_{r_1}$.
\end{enumerate}
In case of (1), the coefficient for $\widetilde{u}$ on the left hand side of \eqref{equ_vicom} is zero but is not on the right one, and therefore it cannot be happened.
Also for (2) case, the coefficient for $\widetilde{u}$
on the left had side of \eqref{equ_vicom} is one, but the one on the left hand side of \eqref{equ_vicom} is less than one since
\[
	0 < a_{r_2} + \cdots + a_{r_u} < 1
\]
and this yields a contradiction. 
\end{proof}

\begin{proposition}\label{proposition_stringconenr}
Let $C_{\bf i}$ be the string cone associated with a reduced word  $\textbf{\textup{i}}$ of the longest element in the Weyl group of $\SL_{n+1}(\C)$. Then the expression~\eqref{equation_stringcone} is non-redundant in $C_{\bf i}$.
\end{proposition}

\begin{proof}
	It is straightforward from Lemma \ref{lemma_fnncc} since the cone generated by $\{ {\bf v}_1, \cdots, {\bf v}_u \}$ is the dual cone of $C_{\bf i}$ by definition.
\end{proof}

Let us fix a \emph{regular dominant} weight $\lambda = \lambda_1 \varpi_1 + \cdots + \lambda_n \varpi_n$, i.e., $\lambda_i > 0$ for all $i$. 
Now, we show that each of string and $\lambda$-inequalities is non-redundant in $\Delta_{\bf i}(\lambda) = C^{\vphantom{\lambda}}_{\bf i} \cap C_{\bf i}^\lambda$. 
Recall that the $\lambda$-cone $C_{\bf i}^\lambda$ has apex whose chamber coordinates are given by $u_{(i)} = \lambda_i$ for $i = 1, \dots, n$ and the other chamber coordinates are zero. 

\begin{proposition}\label{proposition_stringpolytopenr}
Let $\Delta_{\bf i}(\lambda)$ be the string polytope associated to a reduced word  $\textbf{\textup{i}} \in \Sigma_{n+1}$.  
Then the expression in~\eqref{equation_stringcone} and~\eqref{equation_lambdacone} is non-redundant in $\Delta_{\bf i}(\lambda)$.
\end{proposition}

\begin{proof}
	It is enough to show that 
		\begin{itemize}
			\item the inequalities in \eqref{equation_lambdacone} are non-redundant in $C_{\bf i}^\lambda$, 
			\item the apex (the origin) of $C_{\bf i}$ is contained in the interior of $C_{\bf i}^\lambda$, and 
			\item the apex ${\bf u}_0$ of $C_{\bf i}^\lambda$ is contained in the interior of $C_{\bf i}$.
		\end{itemize}
	The first statement follows from the observation that the matrix for the system \eqref{equation_lambdacone} is upper triangular such that all diagonal entries are equal to one.
	 In particular $C_{\bf i}^\lambda$ is simplicial. 
	 Moreover, the second statement can be easily obtained from ~\eqref{equation_change_coordinate} 
         because $\lambda$ is a regular dominant weight.

	For the third statement, we choose any rigorous path $P$. Then the region enclosed by $P$ should contain exactly one unbounded region, say~$\scr{U}_i$, 
	so that the normal vector corresponding to $P$ has entry 
	$1$ on the coordinate $u_{(i)}$. Moreover, since the coordinate $u_{(i)}$ of the apex ${\bf u}_0$ has value $\lambda_i$, we see that
	\[
		\rho_j({\bf u}_0) > 0, \quad \quad j=1,\dots,N
	\]
	because $\lambda_i > 0$ for all $i = 1,\dots,n$ and every coefficient of $\rho_j$ is either one or zero. 
	Therefore the apex of the $\lambda$-cone  is in the interior of the string cone~$C_{\bf i}$.
	This finishes the proof.
\end{proof}

We have a corollary which directly follows from Proposition~\ref{proposition_stringpolytopenr}:
\begin{corollary}
	Let $\bf i$ and $\bf i'$ be two different reduced words in $\Sigma_{n+1}$. Suppose that $|\mathcal{GP}(\mathbf{i})| \neq |\mathcal{GP}(\mathbf{i}')|$. Then the corresponding string polytopes $\Delta_{\bf i}(\lambda)$ and $\Delta_{\bf i'}(\lambda)$ are not combinatorially equivalent for any regular dominant weight $\lambda$. Indeed, two string polytopes $\Delta_{\bf i}(\lambda)$ and $\Delta_{\bf i'}(\lambda)$ have different numbers of facets.
\end{corollary}

\begin{example}
	Let $\mathbf i =(1,2,1,3,2,1)$ and $\mathbf i' = (2,1,3,2,3,1)$. We have seen in~Example~\ref{example_string_cone_GC} and Remark~\ref{rmk_number_of_string_cone_iequalities} that $
	|\mathcal{GP}(\mathbf i)| = 6$ and $ |\mathcal{GP}(\mathbf i')| = 7$.
	Hence the string polytopes $\Delta_{\bf i}(\lambda)$ and $\Delta_{\bf i'}(\lambda)$ have different numbers of facets. 
\end{example}

\section{Simplicial string cones}
\label{secSimplicialStringCones}

	Let ${\bf i} \in \Sigma_{n+1}$ be a reduced word  of $w_0 \in \mathfrak{S}_{n+1}$. 
	In this section, we study how the number of facets of the string cone for an extension (respectively, a contraction) of ${\bf i}$ increases (respectively, decreases) in Proposition~\ref{proposition_strictly_increasing}.
	We denote the number of facets of the string cone $C_{\bf i}$ by $\lVert {\bf i} \rVert$, so that we have $\lVert {\bf i} \rVert = |\mathcal{GP}(\mathbf i)|$ by Proposition~\ref{proposition_stringpolytopenr}.
	Finally, we provide a necessary and sufficient condition on $\mathbf{i}$ such that $C_{\mathbf{i}}$ is simplicial, see Theorem \ref{theorem_simplicial}.
	
	We begin by proving the lemma. 		
	\begin{lemma}\label{lemma_oldpath}
		For any reduced word  ${\bf i} \in \Sigma_{n}$, let $\mathcal{GP}({\bf i})$ be the set of all rigorous paths on $G({\bf i})$. Then there is a canonical inclusion
		\[
				\Psi_\bullet ({\bf i}, s) \colon \mathcal{GP}({\bf i}) \hookrightarrow \mathcal{GP}(E_\bullet (s)({\bf i}))
		\]
		where $E_\bullet(s)$ is the $\bullet$-extension in Definition~\ref{definition_contraction_extension} for each $s = 0, 1, \dots, \frac{n(n-1)}{2}$ and $\bullet = D$ or $A$. Moreover, the image of $\Psi_\bullet ({\bf i}, s)$ is given by 
		\[
			\mathrm{Im} ~\Psi_\bullet ({\bf i}, s) = \{ P \in \mathcal{GP}(E_\bullet (s)({\bf i})) \mid \text{$\node(P)$ does not contain a node lying on $\ell_\bullet$} \}
		\]
		where $\ell_D := \ell_{n+1}$ and $\ell_A := \ell_1$, respectively.
	\end{lemma}
	
	\begin{proof}
		We only provide the proof for the case $\bullet = D$ since the case $\bullet = A$ can be similarly dealt with.  
		For a path $P \in \mathcal{GP}({\bf i})$ with the node expression
		\[
			P = (L_k \rightarrow t_{j_1} \rightarrow \cdots \rightarrow t_{j_r} \rightarrow L_{k+1}), \quad (k \in [n-1])
		\]					
		and a fixed $s \in \{0,1,\dots, \frac{n(n-1)}{2}\}$, we define 
		\begin{equation}\label{equation_embedding}
			\begin{array}{ccccl}\vs{0.1cm}
				\Psi_D ({\bf i}, s) & \colon & \mathcal{GP}({\bf i}) & \rightarrow & \mathcal{GP}(E_D (s)({\bf i})) \\ \vs{0.1cm}
								&  & P                & \mapsto & \Psi_D ({\bf i}, s)(P) \\ \vs{0.1cm}
								&  &			&               & := (L_k \rightarrow t_{\hat{j_1}} \rightarrow \cdots \rightarrow t_{\hat{j_r}} \rightarrow L_{k+1}), 
			\end{array}
			\hs{0.1cm}
			\hat{j_i} = \begin{cases}
				{j_i} & \text{if $j_i \leq s$} \\
				{j_i}+n & \text{if $j_i \geq s + 1$}
			\end{cases}
		\end{equation}
		(See the red paths in Figure \ref{figure_new}.) 
		
		We claim that the map~\eqref{equation_embedding} is well-defined. Namely, the path $\Psi_D ({\bf i}, s)(P)$ is a rigorous path, which follows from the following observation$\colon$ 			
		\begin{itemize}
					\item It respects the orientation of $G(E_D (s)({\bf i}), k)$ since the orientation of the wire $\ell_j$ in $G({\bf i}, k)$ coincides with 
						that of the wire $\ell_j$ in $G(E_D (s)({\bf i}), k)$ for $j = 1, \cdots, n$. 
					\item It does \emph{not} contain any forbidden fragments of $ \Psi_D ({\bf i}, s)(P)$ described in Figure \ref{figure_avoiding} since new patterns 
					appeared on $\ell_{n+1}$ are 
						\begin{figure}[H]
							\scalebox{1}{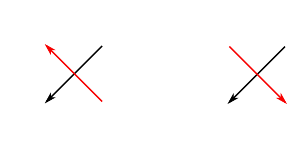}
						\end{figure}					
				\end{itemize}
			 Moreover, the image of $\mathcal{GP}({\bf i})$ under $\Psi_D ({\bf i}, s)$ is explicitly given by
				\[
					O := \{ P \in \mathcal{GP}(E_D (s)({\bf i})) \mid \text{$\node(P)$ does not contain any}~ \underbrace{\text{node lying on $\ell_{n+1}$.}}_{=~t_{s+1}, \dots, t_{s+n}} \}.
				\]
				(Indeed, we can find the inverse of $\Psi_D ({\bf i}, s)(P)$ directly from \eqref{equation_embedding} such that
				\[
					(L_k \rightarrow t_{\hat{j_1}} \rightarrow \cdots \rightarrow t_{\hat{j_r}} \rightarrow L_{k+1}) \mapsto (L_k \rightarrow t_{j_1} \rightarrow \cdots \rightarrow t_{j_r} \rightarrow L_{k+1})
				\]
				from $O$ to $\mathcal{GP}({\bf i})$ as any of $\hat{j_i}$'s is not in $\{s+1, \dots, s+n\}$.)
		This finishes the proof. 
	\end{proof}
	
	Using the relation \eqref{equation_commutation} and Lemma \ref{lemma_oldpath}, one can generalize 
	Lemma \ref{lemma_oldpath} as follows. 
	
	\begin{corollary}\label{corollary_oldpath}
		Let ${\bf i} \in \Sigma_{n+1}$. Then there is canonical inclusion 
		\[
			\widetilde{\Psi}_\bullet({\bf i}) \colon \mathcal{GP}(C_\bullet({\bf i})) \hookrightarrow \mathcal{GP}({\bf i}), \quad \quad \bullet = D ~\text{or} ~A.
		\]
		Also, the image of $\widetilde{\Psi}_\bullet ({\bf i})$ is given by 
		\[
			\mathrm{Im} ~\widetilde{\Psi}_\bullet ({\bf i}) = \{ P \in GP({\bf i}) ~|~ \text{$\node(P)$ does not contain a node lying on $\ell_\bullet$} \}
		\]
		where $\ell_D := \ell_{n+1}$ and $\ell_A := \ell_1$, respectively.
	\end{corollary}
	
	\begin{proof}
		The proof is straightforward since there is a natural identification
		\[
			\phi({\bf i}, {\bf i}') \colon \mathcal{GP}({\bf i}) \rightarrow \mathcal{GP}({\bf i}') 
		\] for any ${\bf i}, {\bf i}' \in \Sigma_{n+1}$ such that ${\bf i} \sim {\bf i}'$.
	\end{proof}
	
	\begin{remark}
		Note that 
		\[
			\widetilde{\Psi}_\bullet({\bf i}) = \Psi_\bullet(C_\bullet({\bf i}), s)
		\]
		when ${\bf i} = E_\bullet(s)(C_\bullet({\bf i}))$ (so that $\mathrm{ind}_\bullet ({\bf i}) = s$).	
	\end{remark}

	For $\bullet = D$ or $A$, we say that a rigorous path $P \in \mathcal{GP}({\bf i})$ is {\em $\bullet$-new} if 
	\[
		P \not \in \mathrm{Im}~ \widetilde{\Psi}_\bullet ({\bf i}) \subset \mathcal{GP}({\bf i}).
	\]
	See Figure \ref{figure_new}; blue paths are $D$-new but red ones are not. 
	
	\begin{figure}[h]
		\scalebox{1}{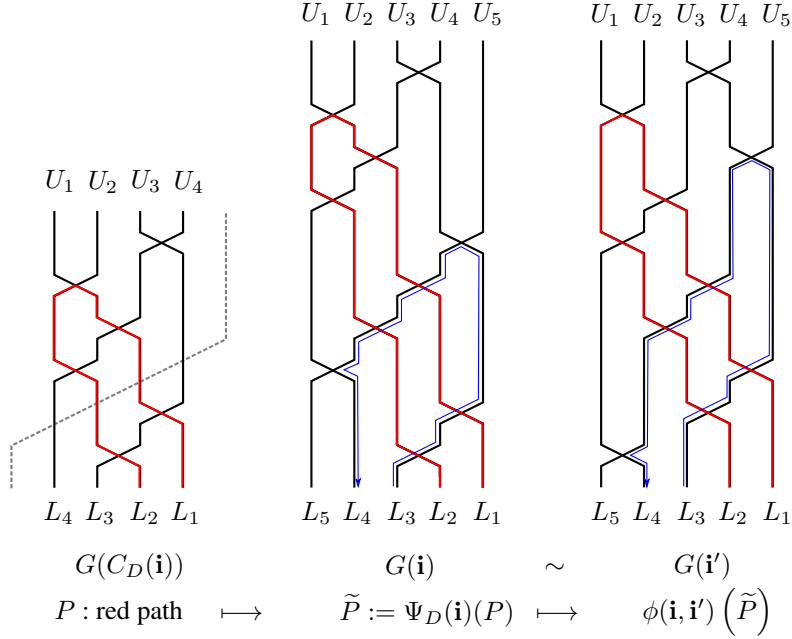}
		\caption{\label{figure_new} Examples of new rigorous paths for ${\bf i}' = (3,1,4,2,1,3,2,4,3,1)$.}
	\end{figure}	
	
	By definition, for a given ${\bf i} \in \Sigma_{n+1}$, 
	a rigorous path in $G({\bf i})$ which has a peak on $\ell_{n+1}$ (on $\ell_1$, respectively) is not $D$-new (not $A$-new, respectively).
	The following series of corollaries are straightforward by Corollary \ref{corollary_oldpath}.
	
	\begin{corollary}\label{corollary_new_iff}
		Given ${\bf i} \in \Sigma_{n+1}$, a rigorous path $P$ in $G({\bf i})$ is $D$-new if and only if at least one node in $\node(P)$ is on $\ell_{n+1}$. Similarly, 
		a rigorous path $P$ in $G({\bf i})$ is $A$-new if and only if at least one node in $\node(P)$ is on $\ell_1$.  
	\end{corollary}

	\begin{corollary}\label{corollary_new_peak}
		Given ${\bf i} \in \Sigma_{n+1}$, a rigorous path $P$ in $G({\bf i})$ is $D$-new \textup{(}$A$-new, respectively\textup{)} if $\node(P)$ contains a peak lying 
		on $\ell_{n+1}$ \textup{(}on $\ell_1$, respectively\textup{)}. 
	\end{corollary}

	The following propositions tell us that we may describe at least $n$ distinct $\bullet$-new rigorous paths in $G({\bf i})$ explicitly for any ${\bf i} \in \Sigma_{n+1}$
	and $\bullet = D$ or $A$.
	
	\begin{proposition}\label{proposition_canonical_D}
		Let ${\bf i} \in \Sigma_{n+1}$ and $t_{j_k}$ be the node at which $\ell_k$ and $\ell_{n+1}$ intersect. 
		For each $k \in [n]$, there exists a rigorous path $P_D({\bf i}, k) \in \mathcal{GP}({\bf i})$ such that 
		\begin{itemize}
			\item it has a unique peak $t_{j_k}$, 
			\item it travels from $\ell_k$ to $\ell_{n+1}$ at $t_{j_k}$, 			
			\item it is below $\ell_{n+1}$, 
			\item with respect to the wire-expression of $P_D({\bf i}, k)\colon$ 
			\[
				\ell_{r_p} \rightarrow \cdots \rightarrow \ell_{r_1} \rightarrow \ell_k \rightarrow \ell_{n+1} \rightarrow \ell_{s_q} \rightarrow \cdots \rightarrow \ell_{s_1} 
				(:= \ell_{i_p+1}),
			\]
			the sequences $r_1, \dots, r_p$ and $s_1, \dots, s_q$ are increasing and decreasing, respectively.
		\end{itemize}		
	\end{proposition}

	\begin{proof}
		Fix $k \in [n]$. We first construct the first half of $P_D({\bf i}, k)$, which is from some $L_{r_p}$ to $t_{j_k}$. Let $P_0 := (\ell_k \rightarrow \ell_{n+1})$
		be an oriented path from $L_k$ to $t_{j_k}$, a portion of $\ell_k$, where the orientation respects the orientation on $G({\bf i}, k)$.
		
		Let $r_1 > k$ be the smallest index such that 
		$\ell_k$ crosses $\ell_{r_1}$ before meeting $\ell_{n+1}$. Then we obtain an oriented path 
		\[
			P_1 := (\ell_{r_1} \rightarrow \ell_k \rightarrow \ell_{n+1} ), 
		\]
		where the orientation respects the orientation of $G({\bf i}, r_1)$. 
		One can check that any pattern (of the second type) in Figure \ref{figure_avoiding} does not occur
		at every node on $P_1$ appeared before $t_{j_k}$ because the forbidden pattern appears only when $\ell_k$ crosses $\ell_t$ (for some $t > k$ with $t \neq r+1$) or 
		$\ell_{r_1}$ crosses $\ell_t$ for some $t > r+1$.  
		See Remark \ref{remark_avoiding} and Figure \ref{figure_proof_construct1}.
		
		Next, let $r_2 >r_1$ be the smallest index such that $\ell_{r_2}$ crosses $P_1$ before meeting $\ell_{n+1}$. Then we obtain an oriented path
		\[
			P_2 := \left(\ell_{r_2} \rightarrow P_1\right) := \begin{cases}
				(\ell_{r_2} \rightarrow \ell_{r_1} \rightarrow \ell_k \rightarrow \ell_{n+1}) & \text{if $\ell_{r_2}$ crosses $\ell_{r_1}$ first,} \\
				(\ell_{r_2} \rightarrow \ell_k \rightarrow \ell_{n+1}) & \text{if $\ell_{r_2}$ crosses $\ell_{k}$ first}. \\
			\end{cases}
		\]		
		where the orientation respects the orientation of $G({\bf i}, r_2)$.
		See Figure \ref{figure_proof_construct1}. One can similarly verify that $P_2$ does not contain any forbidden pattern (of the second type) in Figure \ref{figure_avoiding}.
		Inductively, we obtain an oriented path $P_p := \left(\ell_{r_p} \rightarrow P_{p-1}\right)$ for some $p$ such that there is no wire $\ell_i$ with $i > r_p$ crossing $P_p$ below $\ell_{n+1}$.
		
		To complete the construction of $P_D({\bf i}, k)$, 
		we apply the previous argument to the wire $\ell_{s_1} := \ell_{i_p + 1}$ as follows. If there is no wire $\ell_i$ with $i > s_1$ intersecting $\ell_{s_1}$ below $\ell_{n+1}$, 
		then the path $P_p \rightarrow \ell_{s_1}$ respecting the orientation of $G({\bf i}, r_p)$ does not contain any forbidden pattern (of the first type) in Figure \ref{figure_avoiding}.
		So, $P_p \rightarrow \ell_{s_1}$ is our desired rigorous path. 
		
		If not, let $s_2 > s_1$ be the smallest index such that $\ell_{s_2}$ crosses $\ell_{s_1}$ below $\ell_{n+1}$ and let 
		\[
			{Q}_2 := (\ell_{s_2} \rightarrow \ell_{s_1}).
		\]
		Inductively, we obtain ${Q}_q := (\ell_{s_q} \rightarrow {Q}_{q-1})$ where ${s_q} > s_{q-1}$ is the the smallest index such that $\ell_{s_q}$ crosses ${Q}_{q-1}$ below $\ell_{n+1}$
		and there is no wire $\ell_i$ with $i > q$ crossing ${Q}_q$ below $\ell_{n+1}$. Then, 
		\[
			P_p \rightarrow {Q}_q
		\]
		respecting the orientation of $G({\bf i}, r_p)$ is our desired rigorous path and this completes the proof.
	\end{proof}

	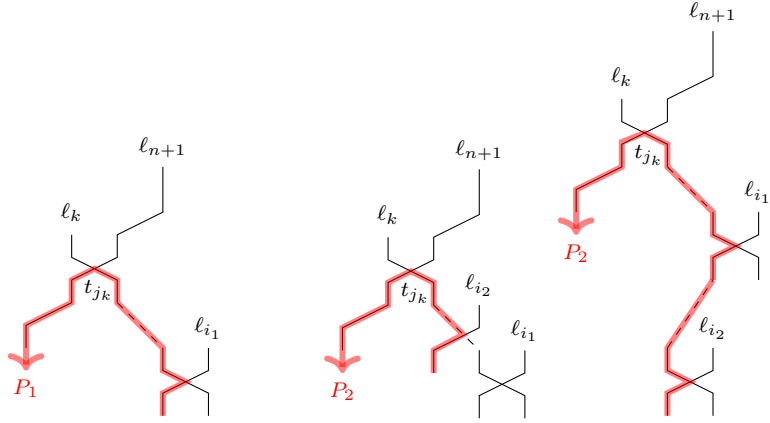
\begin{figure}[h]
		\begin{tikzpicture}[scale = 0.6]
		\tikzset{every node/.style = {font = \footnotesize}}
		\tikzset{red line/.style = {line width=0.5ex, red, semitransparent}}
		
		\draw (1,3) -- (1,3.5)--(2,4)--(2,4.5)--(3,5)--(3,5.5)--(4,6)--(4,7) node[above] {$\ell_{n+1}$};
		\draw (2,5.5) -- (2,5)--(3,4.5)--(3,4);
		
		\draw (4,3) -- (4,2.5)--(5,2)--(5,1.5);
		\draw (5,3)--(5,2.5)--(4,2)--(4,1.5);
		
		\draw[dashed] (3,4)--(4,3);
		
		\node at (2,6) {$\ell_k$};
		\node[below] at (2.6,4.7) {$t_{j_k}$};
		\node[above] at (5,3) {$\ell_{i_1}$};
		
		
		\draw[red line, ->] (4,1.5)--(4,2)--(4.5,2.25)--(4,2.5)--(4,3)--(3,4)--(3,4.5)--(2.5,4.75)--(2,4.5)--(2,4)--(1,3.5)--(1,2.5);
		\node[below, red] at (1,2.5) {$P_1$};
		\end{tikzpicture} \hspace{1cm}
		\begin{tikzpicture}[scale = 0.6]
		\tikzset{every node/.style = {font = \footnotesize}}
		\tikzset{red line/.style = {line width=0.5ex, red, semitransparent}}
		\draw (1,3) -- (1,3.5)--(2,4)--(2,4.5)--(3,5)--(3,5.5)--(4,6)--(4,7) node[above] {$\ell_{n+1}$};
		\draw (2,5.5) -- (2,5)--(3,4.5)--(3,4);
		
		\draw (4,3) -- (4,2.5)--(5,2)--(5,1.5);
		\draw (5,3)--(5,2.5)--(4,2)--(4,1.5);
		
		\draw (4,4)--(4,3.5)--(3,3)--(3,2.5); 
		
		\draw[dashed] (3,4)--(4,3);
		
		\node at (2,6) {$\ell_k$};
		\node[below] at (2.6,4.7) {$t_{j_k}$};
		\node[above] at (5,3) {$\ell_{i_1}$};
		
		\node[above] at (4,4) {$\ell_{i_2}$};
		
		
		\draw[red line, ->] (3,2.5)--(3,3)--(3.65,3.35)--(3,4)--(3,4.5)--(2.5,4.75)--(2,4.5)--(2,4)--(1,3.5)--(1,2.5);
		\node[below, red] at (1,2.5) {$P_2$};
		\end{tikzpicture}
		\begin{tikzpicture}[scale = 0.6]
		\tikzset{every node/.style = {font = \footnotesize}}
		\tikzset{red line/.style = {line width=0.5ex, red, semitransparent}}
		
		\draw (1,3) -- (1,3.5)--(2,4)--(2,4.5)--(3,5)--(3,5.5)--(4,6)--(4,7) node[above] {$\ell_{n+1}$};
		\draw (2,5.5) -- (2,5)--(3,4.5)--(3,4);
		
		\draw (4,3) -- (4,2.5)--(5,2)--(5,1.5);
		\draw (5,3)--(5,2.5)--(4,2)--(4,1.5);
		
		\draw (3,0)--(3,-.5)--(4,-1)--(4,-1.5);
		\draw (4,0)--(4,-0.5)--(3,-1)--(3,-1.5);

		\draw[dashed] (3,0)--(4,1.5);
		\draw[dashed] (3,4)--(4,3);
		
		\node at (2,6) {$\ell_k$};
		\node[below] at (2.6,4.7) {$t_{j_k}$};
		\node[above] at (5,3) {$\ell_{i_1}$};
		\node[above] at (4,0) {$\ell_{i_2}$};
		

		\draw[red line] (4,1.5)--(3,0)--(3,-0.5)--(3.5,-0.75)--(3,-1)--(3,-1.5);
		\draw[red line, ->] (4,1.5)--(4,2)--(4.5,2.25)--(4,2.5)--(4,3)--(3,4)--(3,4.5)--(2.5,4.75)--(2,4.55)--(2,4)--(1,3.5)--(1,2.5);
		\node[below, red] at (1,2.5) {$P_2$};
		
		\end{tikzpicture}
		\caption{\label{figure_proof_construct1} Construction of canonical paths.}
	\end{figure}	
		
	Similarly, we obtain the following $A$-type analogue of Proposition \ref{proposition_canonical_D} where we omit the proof. 

	\begin{proposition}\label{proposition_canonical_A}
		Let ${\bf i} \in \Sigma_{n+1}$ and $t_{j_k}$ be the node at which $\ell_1$ and $\ell_{k+1}$ intersect for $k \in [n]$.  
		Then there exists a rigorous path $P_A({\bf i}, k) \in \mathcal{GP}({\bf i})$ such that 
		\begin{itemize}
			\item it has a unique peak $t_{j_k}$, 
			\item it travels from $\ell_{1}$ to $\ell_{k+1}$ at $t_{j_k}$, 			
			\item it is below $\ell_{1}$, 
			\item with respect to the wire-expression of $P_A({\bf i},k)$: 
			\[
				\ell_{s_1} \rightarrow \cdots \rightarrow \ell_{s_q} \rightarrow \ell_1 \rightarrow \ell_{k+1} \rightarrow \ell_{r_1} \rightarrow \cdots \rightarrow \ell_{r_p},
			\]
			the sequences $r_1, \dots, r_p$ and $s_1, \dots, s_q$ are increasing and decreasing, respectively.
		\end{itemize}		
	\end{proposition}

	We call a path $P_\bullet({\bf i}, k)$ constructed in Propositions \ref{proposition_canonical_D} and  \ref{proposition_canonical_A}
	{\em $\bullet$-canonical}. 
	Note that every $\bullet$-canonical path is $\bullet$-new by Corollary \ref{corollary_new_peak}. 
	
	\begin{example}\label{example_canonical_paths}
		Let ${\bf i} = (4,3,2,1,4,2,3,2,4,3) \in \Sigma_5$. 
		Then there are four $D$-canonical paths:	
		\begin{align*}
		P_D(\mathbf i,1) &= (\ell_4 \to \ell_2 \to \ell_1 \to \ell_5), &
		P_D(\mathbf i, 2) &= (\ell_4 \to \ell_2 \to \ell_5), \\
		P_D(\mathbf i, 3) &= (\ell_4 \to \ell_3 \to \ell_5), &
		P_D(\mathbf i,4) &= (\ell_4 \to \ell_5).
		\end{align*}
		These paths are given in Figure \ref{figure_D_canonical_path}.
		
		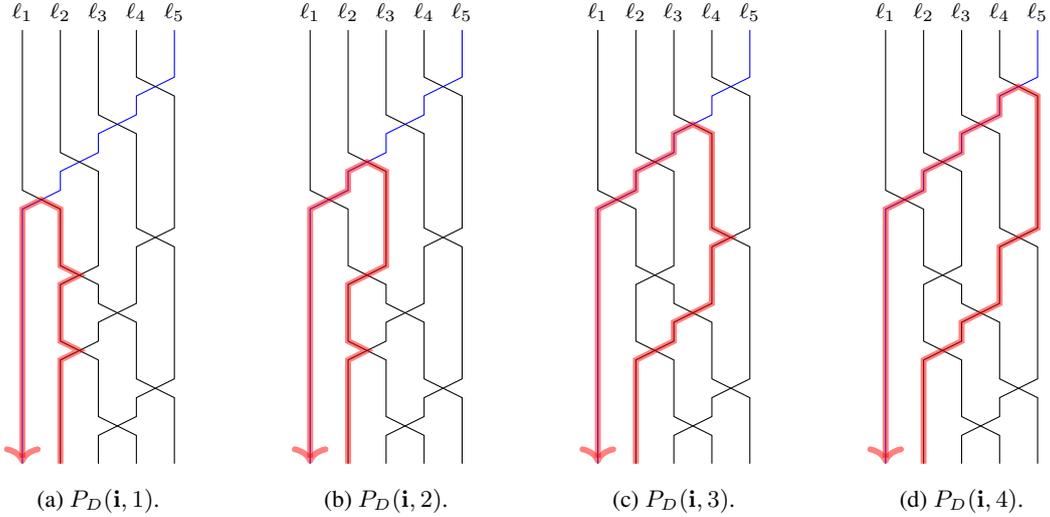
\begin{figure}[h]
			\begin{subfigure}[c]{0.23\textwidth}
				\centering
			\begin{tikzpicture}[scale = 0.5]
			\tikzset{every node/.style = {font = \footnotesize}}
			\tikzset{red line/.style = {line width=0.5ex, red, semitransparent}}
			
			\draw[blue] (0,0)--(0,6.75)--(1,7.25)--(1,7.75)--(2,8.25)--(2,8.75)--(3,9.25)--(3,9.75)--(4,10.25)--(4,11.5); 
			\draw (1,0)--(1,2.75)--(2,3.25)--(2,3.75)--(3,4.25)--(3,5.75)--(4,6.25)--(4,9.75)--(3,10.25)--(3,11.5); 
			\draw (2,0)--(2,0.75)--(3,1.25)--(3,1.75)--(4,2.25)--(4,5.75)--(3,6.25)--(3,8.75)--(2,9.25)--(2,11.5); 
			\draw (3,0)--(3,0.75)--(2,1.25)--(2,2.75)--(1,3.25)--(1,4.75)--(2,5.25)--(2,7.75)--(1,8.25)--(1,11.5); 
			\draw (4,0)--(4,1.75)--(3,2.25)--(3,3.75)--(2,4.25)--(2,4.75)--(1,5.25)--(1,6.75)--(0,7.25)--(0,11.5); 
			
			\node[above] at (0,11.5) {$\ell_1$};
			\node[above] at (1,11.5) {$\ell_2$};
			\node[above] at (2,11.5) {$\ell_3$};
			\node[above] at (3,11.5) {$\ell_4$};
			\node[above] at (4,11.5) {$\ell_5$};
			
			\draw[red line, ->] (1,0)--(1,2.75)--(1.5,3)--(1,3.25)--(1,4.75)--(1.5,5)--(1,5.25)--(1,6.75)--(0.5,7)--(0,6.75)--(0,0);
			\end{tikzpicture}			
			\caption{$P_D(\mathbf i, 1)$.}
		\end{subfigure}
	\begin{subfigure}[c]{0.23\textwidth}
		\centering
			\begin{tikzpicture}[scale = 0.5]
			\tikzset{every node/.style = {font = \footnotesize}}
			\tikzset{red line/.style = {line width=0.5ex, red, semitransparent}}
			
			\draw[blue] (0,0)--(0,6.75)--(1,7.25)--(1,7.75)--(2,8.25)--(2,8.75)--(3,9.25)--(3,9.75)--(4,10.25)--(4,11.5); 
			\draw (1,0)--(1,2.75)--(2,3.25)--(2,3.75)--(3,4.25)--(3,5.75)--(4,6.25)--(4,9.75)--(3,10.25)--(3,11.5); 
			\draw (2,0)--(2,0.75)--(3,1.25)--(3,1.75)--(4,2.25)--(4,5.75)--(3,6.25)--(3,8.75)--(2,9.25)--(2,11.5); 
			\draw (3,0)--(3,0.75)--(2,1.25)--(2,2.75)--(1,3.25)--(1,4.75)--(2,5.25)--(2,7.75)--(1,8.25)--(1,11.5); 
			\draw (4,0)--(4,1.75)--(3,2.25)--(3,3.75)--(2,4.25)--(2,4.75)--(1,5.25)--(1,6.75)--(0,7.25)--(0,11.5); 
			
			\node[above] at (0,11.5) {$\ell_1$};
			\node[above] at (1,11.5) {$\ell_2$};
			\node[above] at (2,11.5) {$\ell_3$};
			\node[above] at (3,11.5) {$\ell_4$};
			\node[above] at (4,11.5) {$\ell_5$};
			
			\draw[red line, ->] (1,0)--(1,2.75)--(1.5,3)--(1,3.25)--(1,4.75)--(2,5.25)--(2,7.75)--(1.5,8)--(1,7.75)--(1,7.25)--(0,6.75)--(0,0);
			\end{tikzpicture}			
			\caption{$P_D(\mathbf i, 2)$.}
		\end{subfigure}
	\begin{subfigure}[c]{0.23\textwidth}
		\centering
			\begin{tikzpicture}[scale = 0.5]
			\tikzset{every node/.style = {font = \footnotesize}}
			\tikzset{red line/.style = {line width=0.5ex, red, semitransparent}}
			
			\draw[blue] (0,0)--(0,6.75)--(1,7.25)--(1,7.75)--(2,8.25)--(2,8.75)--(3,9.25)--(3,9.75)--(4,10.25)--(4,11.5); 
			\draw (1,0)--(1,2.75)--(2,3.25)--(2,3.75)--(3,4.25)--(3,5.75)--(4,6.25)--(4,9.75)--(3,10.25)--(3,11.5); 
			\draw (2,0)--(2,0.75)--(3,1.25)--(3,1.75)--(4,2.25)--(4,5.75)--(3,6.25)--(3,8.75)--(2,9.25)--(2,11.5); 
			\draw (3,0)--(3,0.75)--(2,1.25)--(2,2.75)--(1,3.25)--(1,4.75)--(2,5.25)--(2,7.75)--(1,8.25)--(1,11.5); 
			\draw (4,0)--(4,1.75)--(3,2.25)--(3,3.75)--(2,4.25)--(2,4.75)--(1,5.25)--(1,6.75)--(0,7.25)--(0,11.5); 
			
			\node[above] at (0,11.5) {$\ell_1$};
			\node[above] at (1,11.5) {$\ell_2$};
			\node[above] at (2,11.5) {$\ell_3$};
			\node[above] at (3,11.5) {$\ell_4$};
			\node[above] at (4,11.5) {$\ell_5$};
			
			\draw[red line, ->] (1,0)--(1,2.75)--(2,3.25)--(2,3.75)--(3,4.25)--(3,5.75)--(3.5,6)--(3,6.25)--(3,8.75)--(2.5,9)--(2,8.75)--(2,8.25)--(1,7.75)--(1,7.25)--(0,6.75)--(0,0);
			\end{tikzpicture}			
			\caption{$P_D(\mathbf i, 3)$.}
		\end{subfigure}
	\begin{subfigure}[c]{0.23\textwidth}
		\centering
			\begin{tikzpicture}[scale = 0.5]
			\tikzset{every node/.style = {font = \footnotesize}}
			\tikzset{red line/.style = {line width=0.5ex, red, semitransparent}}
			
			\draw[blue] (0,0)--(0,6.75)--(1,7.25)--(1,7.75)--(2,8.25)--(2,8.75)--(3,9.25)--(3,9.75)--(4,10.25)--(4,11.5); 
			\draw (1,0)--(1,2.75)--(2,3.25)--(2,3.75)--(3,4.25)--(3,5.75)--(4,6.25)--(4,9.75)--(3,10.25)--(3,11.5); 
			\draw (2,0)--(2,0.75)--(3,1.25)--(3,1.75)--(4,2.25)--(4,5.75)--(3,6.25)--(3,8.75)--(2,9.25)--(2,11.5); 
			\draw (3,0)--(3,0.75)--(2,1.25)--(2,2.75)--(1,3.25)--(1,4.75)--(2,5.25)--(2,7.75)--(1,8.25)--(1,11.5); 
			\draw (4,0)--(4,1.75)--(3,2.25)--(3,3.75)--(2,4.25)--(2,4.75)--(1,5.25)--(1,6.75)--(0,7.25)--(0,11.5); 
			
			\node[above] at (0,11.5) {$\ell_1$};
			\node[above] at (1,11.5) {$\ell_2$};
			\node[above] at (2,11.5) {$\ell_3$};
			\node[above] at (3,11.5) {$\ell_4$};
			\node[above] at (4,11.5) {$\ell_5$};
			
			\draw[red line, ->] (1,0)--(1,2.75)--(2,3.25)--(2,3.75)--(3,4.25)--(3,5.75)--(4,6.25)--(4,9.75)--(3.5,10)--(3,9.75)--(3,9.25)--(2,8.75)--(2,8.25)--(1,7.75)--(1,7.25)--(0,6.75)--(0,0);
			\end{tikzpicture}
			\caption{$P_D(\mathbf i, 4)$.}
			\end{subfigure}
			\caption{\label{figure_D_canonical_path} $D$-canonical paths. }
		\end{figure}
	\end{example}
	
		The following lemma is a direct consequence of the proofs of Propositions \ref{proposition_canonical_D} and \ref{proposition_canonical_A}:
	\begin{lemma}\label{lemma_canoncial_paths_index_zero}
		Let $\mathbf i \in \Sigma_{n+1}$. Then we have the following.
		\begin{enumerate}
			\item If $\textup{ind}_{D}(\mathbf i) =0$, then for $k \in [n]$, the $D$-canonical  path $P_{D}(\mathbf i, k)$ is given by
			\[
			\ell_k \to \ell_{n+1} \to \ell_{k+1}
			\]
			with respect to the wire-expression.
			\item If $\textup{ind}_A(\mathbf i ) = 0$, then for $k \in [n]$, the $A$-canonical  path $P_{A}(\mathbf i,k)$ is given by
			\[
			\ell_k \to \ell_1 \to \ell_{k+1}
			\]
			with respect to the wire-expression.			
		\end{enumerate}
	\end{lemma}

	We are now ready to prove the main proposition of this section.
	
	\begin{proposition}\label{proposition_strictly_increasing}
		For any ${\bf i} \in \Sigma_{n+1}$ and any $\bullet = D, A$, we have 
		\begin{equation}\label{equation_increasing}
			\lVert {\bf i} \rVert \geq \lVert C_\bullet ({\bf i}) \rVert + n.
		\end{equation}
		Moreover, the inequality~\eqref{equation_increasing} is strict for some $\bullet$ if 
				$\mathrm{ind}_A({\bf i}) \cdot \mathrm{ind}_D({\bf i}) > 0$
	\end{proposition}

	\begin{proof}
		In Propositions~\ref{proposition_canonical_D} and~\ref{proposition_canonical_A}, we have found  $n$ distinct $\bullet$-canonical paths in $\mathcal{GP}({\bf i})$, which in particular are $\bullet$-new for $\bullet = D$ and $A$ respectively. 
		The inequality \eqref{equation_increasing} then immediately follows. 
		
		
		It remains to verify the second statement. 
		Assume that $\mathrm{ind}_D({\bf i}) > 0$ and $\mathrm{ind}_A({\bf i}) > 0$. 
		We shall construct a rigorous path that is $\bullet$-new but \emph{not} $\bullet$-canonical in $\mathcal{GP}({\bf i})$ for some $\bullet = A, D$. 
		The existence of such a path would imply that
		the inequality \eqref{equation_increasing} is strict. 
		
		The following three cases are separately taken into account.
		\begin{itemize}
			\item (Case I.) There exists at least one pair $(s, b)$ of indices satisfying $1 < s, \, b < n$ and the following$\colon$
			 \begin{itemize}
			  \item $\ell_s$ intersects $\ell_{n+1}$ below $\ell_1$, and 
			  \item $\ell_b$ intersects $\ell_{1}$ below $\ell_{n+1}$.
			  \end{itemize}
			\item (Case II.) There does not exist any $s$ satisfying $1 < s < n$ and $\ell_s$ intersects $\ell_{n+1}$ below $\ell_1$.
			\item (Case III.) There does not exist any $b$ satisfying $1 < b < n$ and $\ell_b$ intersects $\ell_{1}$ below $\ell_{n+1}$.
		\end{itemize}

		For (Case I), let $s_0$ (resp. $b_0$) be the smallest (resp. largest) index among $s$ (resp. $b$) satisfying $1 < s < n$ and $\ell_s$ meets $\ell_{n+1}$ below $\ell_{1}$ (resp. $1 < b < n$ and $\ell_b$ meets $\ell_{1}$ below $\ell_{n+1}$). Then we deal with two sub cases:

				\begin{figure}[h]
					\begin{subfigure}[b]{0.3\textwidth}
						\centering
\begin{tikzpicture}[scale = 0.45]
\tikzset{every node/.style = {font = \footnotesize}}
\tikzset{red line/.style = {line width=0.5ex, red, semitransparent}}

\draw (8,0)--(8,4)--(0,8)--(0,11); 
\draw (0,0)--(0,4)--(8,8)--(8,11); 

\draw (3,0)--(3,2);
\draw[dashed] (3,2)--(3,3)--(2,4);
\draw (2,4)--(2,4.5)--(1,5)--(1,5.5);
\draw[dashed] (1,5.5)--(1,6.5);
\draw (1,6.5)--(1,7)--(2,7.5)--(2,8);
\draw[dashed] (2,8)--(3,9);
\draw (3,9)--(3,10)--(4,10.5)--(4,11);

\draw (5,0)--(5,2);
\draw[dashed] (5,2)--(5,3)--(6,4);
\draw (6,4)--(6,4.5)--(7,5)--(7,5.5);
\draw[dashed] (7,5.5)--(7,6.5);
\draw (7,6.5)--(7,7)--(6,7.5)--(6,8);
\draw[dashed] (6,8)--(4,9);
\draw (4,9)--(4,10)--(3,10.5)--(3,11);

\draw[red line, ->] (5,0)--(5,2)--(5,3)--(6,4)--(6,4.5)--(7,5)--(7,5.5)--(7,6.5)--(7,7)--(6.5,7.25)--(4,6)--(1.5,7.25)--(1,7)--(1,6.5)--(1,5.5)--(1,5)--(2,4.5)--(2,4)--(3,3)--(3,2)--(3,0);

\node[below] at (8,0) {$\ell_1$};
\node[below] at (0,0) {$\ell_{n+1}$};
\node[below] at (3,0) {$\ell_{s_0}$}; 
\node[below] at (5,0) {$\ell_{b_0}$};

\end{tikzpicture}
						\caption{Case I-1: $s_0 > b_0 \, (s_0 = b_0+1).$}\label{labcasei-1}
					\end{subfigure}
				\begin{subfigure}[b]{0.6\textwidth}
					\centering
\begin{tikzpicture}[scale = 0.47]
\tikzset{every node/.style = {font = \footnotesize}}
\tikzset{red line/.style = {line width=0.5ex, red, semitransparent}}

\draw (8,0)--(8,4)--(0,8)--(0,9); 
\draw (0,0)--(0,4)--(8,8)--(8,9); 

\draw (5,0)--(5,1)--(4,1.5)--(4,2);
\draw[dashed] (4,2)--(3,3)--(2,4);
\draw (2,4)--(2,4.5)--(1,5)--(1,5.5);
\draw[dashed] (1,5.5)--(1,6.5);
\draw (1,6.5)--(1,7)--(2,7.5)--(2,8);

\draw (4,0)--(4,1)--(5,1.5)--(5,2);
\draw[dashed] (5,2)--(5,4.5);
\draw (5,4.5)--(5,5)--(6,5.5)--(6,6);
\draw[dashed] (6,6)--(7,6.5);
\draw (7,6.5)--(7,7)--(6,7.5)--(6,8);

\draw[blue] (3,0)--(3,2);
\draw[blue, dashed] (3,2)--(1.5,3.5);
\draw[blue] (1.5,3.5)--(1.5,4)--(0,4.75);

\draw[purple] (6,0)--(6,1);
\draw[red, dashed] (6,1)--(6,3)--(7,3.5);
\draw[red] (7,3.5)--(7,4)--(8,4.5);

\draw[purple] (6,0)--(6,1);
\draw[purple, dashed] (6,1)--(6,3)--(7,3.5);
\draw[purple] (7,3.5)--(7,4)--(8,4.5);

\node[below] at (5,0) {$\ell_{s_0}$};
\node[below] at (4,0) {$\ell_{b_0}$};
\node[below, blue] at (2.8,0) {$\ell_{b_0+1}$};
\node[below, purple] at (6.2,0) {$\ell_{s_0-1}$};

\node[below] at (0,0) {$\ell_{n+1}$}; 
\node[below] at (8,0) {$\ell_1$};

\draw[red line, ->] (6,0)--(6,1)--(6,3)--(7,3.5)--(7,4)--(7.5,4.25)--(4,6)--(1.5,4.75)--(2,4.5)--(2,4)--(3,3)--(4,2)--(4,1.5)--(5,1)--(5,0);

\end{tikzpicture}
\begin{tikzpicture}[scale = 0.5]
\tikzset{every node/.style = {font = \footnotesize}}
\tikzset{red line/.style = {line width=0.5ex, red, semitransparent}}

\draw (8,0)--(8,4)--(0,8)--(0,9); 
\draw (0,0)--(0,4)--(8,8)--(8,9); 

\draw (5,0)--(5,1)--(4,1.5)--(4,2);
\draw[dashed] (4,2)--(3,3)--(2,4);
\draw (2,4)--(2,4.5)--(1,5)--(1,5.5);
\draw[dashed] (1,5.5)--(1,6.5);
\draw (1,6.5)--(1,7)--(2,7.5)--(2,8);

\draw (4,0)--(4,1)--(5,1.5)--(5,2);
\draw[dashed] (5,2)--(5,4.5);
\draw (5,4.5)--(5,5)--(6,5.5)--(6,6);
\draw[dashed] (6,6)--(7,6.5);
\draw (7,6.5)--(7,7)--(6,7.5)--(6,8);

\draw[blue] (3,0)--(3,2);
\draw[blue, dashed] (3,2)--(1.5,3.5);
\draw[blue] (1.5,3.5)--(1.5,4)--(0,4.75);

\draw[purple] (6,0)--(6,1);
\draw[red, dashed] (6,1)--(6,3)--(7,3.5);
\draw[red] (7,3.5)--(7,4)--(8,4.5);

\draw[purple] (6,0)--(6,1);
\draw[purple, dashed] (6,1)--(6,3)--(7,3.5);
\draw[purple] (7,3.5)--(7,4)--(8,4.5);

\node[below] at (5,0) {$\ell_{s_0}$};
\node[below] at (4,0) {$\ell_{b_0}$};
\node[below, blue] at (2.8,0) {$\ell_{b_0+1}$};
\node[below, purple] at (6.2,0) {$\ell_{s_0-1}$};

\node[below] at (0,0) {$\ell_{n+1}$}; 
\node[below] at (8,0) {$\ell_1$};

\draw[line width = 0.5ex, blue, semitransparent, ->] (4,0)--(4,1)--(5,1.5)--(5,2)--(5,4.5)--(5,5)--(5.5,5.25)--(4,6)--(0.75,4.375)--(1.5,4)--(1.5,3.5)--(3,2)--(3,0);
\end{tikzpicture}
					\caption{Case I-2: $s_0 < b_0$.}\label{figure_22p1p2}
				\end{subfigure}
					\caption{\label{figure_5_7_1} $\bullet$-new paths: Case I.}
				\end{figure}
				
		\begin{itemize}
		\item $(\mbox{Case I-1}; s_0 > b_0.)$  Due to the ``smallest'' condition on $s_0$, two wires $\ell_{s_0}$ and $\ell_{s_0-1}$ meet at a node above both $\ell_1$ and $\ell_{n+1}$. 
		Moreover, $\ell_{s_0-1}$ intersects $\ell_{n+1}$ above $\ell_1$.
		In other words, $\ell_{s_0-1}$ intersects $\ell_1$ below $\ell_{n+1}$. 
		Thus, $s_0 = b_0+1$ because of the ``largest'' condition on $b_0$.
		
		We consider a path $P$ (red path in Figure \ref{labcasei-1})
		in $G({\bf i}, b_0)$ that is expressed as 
		\[
			P := (\ell_{b_0} \rightarrow \ell_{n+1} \rightarrow \ell_1 \rightarrow \ell_{b_0+1}). 
		\]
		The path $P$ is a rigorous path since it does not possess any forbidden fragment in Figure \ref{figure_avoiding} by the following reasons (see Figure \ref{figure_5_7_1}):
		\begin{itemize}
			\item No forbidden fragment appears at any node on the wire $\ell_{b_0}$ since $b_0$ is the largest index of a wire having upward orientation
			(see Remark \ref{remark_avoiding}).
			\item When $P$ goes down along the wire $\ell_{n+1}$, if a forbidden fragment occurs, then it implies that some wire $\ell_k$ with downward orientation (or equivalently 
			with $k>b_0$) intersects
			the wire $\ell_1$ below $\ell_{n+1}$. This contradicts the ``largest assumption'' of $b_0$.
			\item On $\ell_1$ part of the path $P$, a forbidden fragment appears only when there exists some wire $\ell_k$ with upward orientation such that 
			$\ell_k$ intersects $\ell_1$ above $\ell_{n+1}$. That is, $k < b_0$ (by the orientation) and $\ell_k$ intersects $\ell_{n+1}$ below $\ell_1$.
			This contradicts the ``smallest assumption'' of $s_0$. 
			\item When $P$ goes down along $\ell_{s_0}$ (= $\ell_{b_0 + 1}$), a forbidden fragment appears if there is some $\ell_k$ with downward orientation 
			and $k < s_0$. However, $s_0$ is the smallest index of a wire in $G({\bf i}, b_0)$ with downward orientation. So, no forbidden fragment exists.
		\end{itemize}
		Note that $P$ is \emph{not} $\bullet$-canonical for any $\bullet = D, A$  
		since it is below neither $\ell_{n+1}$ nor $\ell_1$, while $P$ is both $D$ and $A$-new.  
		
		\item $(\mbox{Case I-2}; s_0 < b_0.)$ In this case, $\ell_{s_0}$ and $\ell_{b_0}$ intersect below both $\ell_1$ and $\ell_{n+1}$. 
		We can construct two rigorous paths $P_1 \in G({\bf i}, s_0 - 1)$ and $P_2 \in G({\bf i}, b_0)$ (red and blue ones in Figure~\ref{figure_22p1p2}) 
		where 
		\[
			P_1 := (\ell_{s_0-1} \rightarrow \cdots \rightarrow \ell_1 \rightarrow \ell_{n+1} \rightarrow \cdots \rightarrow \ell_{s_0}), \quad 
			P_2 := (\ell_{b_0} \rightarrow \cdots \rightarrow \ell_1 \rightarrow \ell_{n+1} \rightarrow \cdots \rightarrow \ell_{b_0+1}).
		\]
		
		The construction of each $P_i$ is similar to the canonical path construction described in Propositions \ref{proposition_canonical_D} and~\ref{proposition_canonical_A}.
		For $P_1$, we start with the wire $\ell_{s_0 - 1}$ with upward orientation in $G({\bf i}, s_0 - 1)$. Similar to Case I-1, we can easily see that no forbidden fragment 
		appears at any node on $\ell_{s_0 - 1}$. \vs{0.2cm}
		
		\begin{itemize}
			\item On the path $\ell_{s_0 - 1} \rightarrow \ell_1$ in $G({\bf i}, s_0 - 1)$, if some forbidden fragment appears at some node on $\ell_1$ below $\ell_{n+1}$, this implies that 
			there is some wire $\ell_{k_1}$ with upward orientation (i.e., $k_1 < s_0 - 1$) such that $\ell_{k_1}$ intersects $\ell_1$ below $\ell_{n+1}$. It follows that $\ell_{s_0 - 1}$ and 
			$\ell_k$ should intersect below $\ell_1$ so that 
			we obtain a path in $G({\bf i}, s_0 - 1)$:
			\[
				\ell_{s_0 - 1} \rightarrow \ell_{k_1} \rightarrow \ell_1.
			\]
			Inductively, whenever a forbidden fragment appears at some node on the path obtained in the previous step, we may insert a new wire to the path so that eventually get
			a path $\widetilde{P}_1$ in $G({\bf i}, s_0 - 1)$ such that 
			\[
				\widetilde{P}_1 := (\ell_{s_0 - 1} \rightarrow \ell_{k_1} \rightarrow \cdots \rightarrow \ell_{k_s} \rightarrow \ell_1)
			\]
			with no forbidden fragment. By the similar procedure, we can complete the construction 
			of our desired rigorous path $P_1$ by inserting wires (having downward orientation) 
			to the path $\widetilde{P}_1 \rightarrow \ell_{n+1} \rightarrow \ell_{s_0}$ in $G({\bf i}, s_0 - 1)$. 
		\end{itemize}
		\vs{0.2cm}
		The construction of $P_2$ is similar to the case of $P_1$.
		We note that both $P_1$ and $P_2$ have the unique peak $\ell_1 \cap \ell_{n+1}$ and are below $\ell_1$ and $\ell_{n+1}$. 
		Therefore, for any $\bullet = D, A$, either $P_1$ or $P_2$ is not 
		$\bullet$-canonical because there is exactly one $\bullet$-canonical path having the node at which $\ell_1$ and $\ell_{n+1}$ intersect as a peak. 
		\end{itemize}
		\vs{0.3cm}


	For (Case II), $\ell_1$ and $\ell_{n+1}$ must meet on the first column of $G({\bf i})$. Let $r$ be the largest index in $\{2,\dots, n\}$ such that $\ell_r$ has a node below $\ell_1$
	obtained by crossing some $\ell_k$ with $r > k$. Such an $r$ exists because $\mathrm{ind}_D({\bf i}) \cdot \mathrm{ind}_A({\bf i}) > 0$. 
	
	We construct a rigorous path $Q$ in $G({\bf i}, r)$ as follows. First, let us start with the path $Q_0 := (\ell_r \rightarrow \ell_1)$ in $G({\bf i}, r)$. Suppose that 
	some forbidden fragment appears at a node, say $t_{j_0}$, on $Q_0$ below $\ell_{n+1}$, that is, there exists a wire $\ell_{k_1}$ with upward orientation (i.e., $k_1 < r$) 
	which intersects $\ell_1$ below $\ell_{n+1}$ and on the left of $\ell_r$ (with respect to the orientation of $\ell_r$).
	We may further assume that $t_{j_0}$ is the nearest one from $\ell_r \cap \ell_1$ among such nodes on $\ell_1$
	(see Figure \ref{figure_5_7_2}).
	Then, we obtain a new path 
	\[
		Q_1 := (\ell_r \rightarrow \ell_1 \rightarrow \ell_{k_1})
	\]	
	on $G({\bf i},r)$.
	Similarly if $Q_1$ contains some forbidden fragment at some node, say $t_{j_1}$, nearest from $\ell_1 \cap \ell_{k_1}$, then we get a new path on $G({\bf i}, r)$:
	\[
		Q_2 := (\ell_r \rightarrow \ell_1 \rightarrow \ell_{k_1} \rightarrow \ell_{k_2}	)
	\]
	where $\ell_{k_1} \cap \ell_{k_2} = t_{j_1}$. Thus we inductively get a path in $G({\bf i}, r)$ 
	\[
		Q_s :=( \ell_r \rightarrow \ell_1 \rightarrow \cdots \rightarrow \ell_{k_s} )
	\]
	such that 
	\begin{itemize}
		\item $Q_s$ does not contain any forbidden fragment at any node below $\ell_{n+1}$ and on the left of $\ell_{r+1}$ with respect to the (downward) orientation of $\ell_{r+1}$,  
		\item $t_{j_{s-1}} = \ell_{k_{s-1}} \cap \ell_{k_s}$ is below $\ell_{n+1}$ and on the left of $\ell_{r+1}$, and  
		\item $s$ is maximal.
	\end{itemize}	
	We deal with two sub cases:

	\begin{itemize}
		\item $(\mbox{Case II-1.})$ Two wires $\ell_{k_s}$ and $\ell_{r+1}$ do not meet below $\ell_{n+1}$. Then we can construct a rigorous path in $G({\bf i}, r)$:
		\[
			Q := (\ell_r \rightarrow \ell_1 \rightarrow \ell_{k_1} \rightarrow \cdots \rightarrow \ell_{k_s} \rightarrow \ell_{n+1} \rightarrow \ell_{m_1} \rightarrow \cdots \rightarrow \ell_{m_t}
			\rightarrow \ell_{r+1}), 
		\]
		where $\ell_{m_1}, \dots, \ell_{m_t}$ are wires with downward orientation obtained in exactly the same way. 
		See Figure~\ref{fcase21}.
		\item $(\mbox{Case II-2.})$ Two wires $\ell_{k_s}$ and $\ell_{r+1}$ meet below $\ell_{n+1}$. Then we have a rigorous path:
		\[
			{Q} := (\ell_r \rightarrow \ell_1 \rightarrow \ell_{k_1} \rightarrow \cdots \rightarrow \ell_{k_s} \rightarrow \ell_{r+1}).
		\]
		See Figure \ref{fcase22}.
	\end{itemize}
	\vs{0.2cm}
	In any case, the path $Q$ is $A$-new but \emph{not} $A$-canonical since it is not below $\ell_1$, and hence $\lVert {\bf i} \rVert > \lVert C_A ({\bf i}) \rVert + n$.


		\begin{figure}[H]
			\begin{subfigure}[b]{0.49\textwidth}
				\centering
\begin{tikzpicture}[scale = 0.57]
\tikzset{every node/.style = {font = \footnotesize}}
\tikzset{red line/.style = {line width=0.5ex, red, semitransparent}}
\tikzset{blue line/.style = {line width=0.5ex, blue, semitransparent}}

\draw (0,0)--(0,6.25)--(5,8.75)--(5,10);
\draw (5,0)--(5,4.25)--(0,6.75)--(0,10);

\draw (1,0)--(1,5.75)--(2,6.25);
\draw[dashed] (2,6.25)--(2,6.75);
\draw (2,6.75)--(1,7.25)--(1,8.25);
\draw (2,0)--(2,1.75)--(3,2.25)--(3,4.75)--(4,5.25);
\draw[dashed] (4,5.25)--(4,7.75);
\draw (4,7.75)--(3,8.25)--(3,9.25);
\draw (3,0)--(3,1.75)--(2,2.25)--(2,5.25)--(3,5.75);
\draw[dashed] (3,5.75)--(3,7.25);
\draw (3,7.25)--(2,7.75)--(2,8.75);

\node[below] at (0,0) {$\ell_{n+1}$};
\node[below] at (1,0) {$\ell_{r+1}$};
\node[below] at (2,0) {$\ell_r$};
\node[below] at (3,0) {$\ell_{k_s}$};
\node[below] at (5,0) {$\ell_1$};

\draw[red line, ->] (2,0)--(2,1.75)--(3,2.25)--(3,4.75)--(3.5,5)--(2.5,5.5)--(3,5.75)--(3,7.25)--(2.5,7.5)--(1.5,7)--(2,6.75)--(2,6.25)--(1,5.75)--(1,0);

\end{tikzpicture}
%
%
%
%
%
%
\caption{Case II-1.}\label{fcase21}
			\end{subfigure}
		\begin{subfigure}[b]{0.49\textwidth}
			\centering
			\begin{tikzpicture}[scale = 0.6]
			\tikzset{every node/.style = {font = \footnotesize}}
			\tikzset{red line/.style = {line width=0.5ex, red, semitransparent}}
			\tikzset{blue line/.style = {line width=0.5ex, blue, semitransparent}}
			
			\draw (0,0)--(0,6.25)--(5,8.75)--(5,10);
			\draw (5,0)--(5,4.25)--(0,6.75)--(0,10);
			
			\draw (1,0)--(1,5.75)--(2,6.25)--(3,6.75)--(3,7.25)--(2,7.75)--(2,8.75);
\draw (2,0)--(2,1.75)--(3,2.25)--(3,4.75)--(4,5.25);
\draw[dashed] (4,5.25)--(4,7.75);
\draw (4,7.75)--(3,8.25)--(3,9.25);
			\draw (3,0)--(3,1.75)--(2,2.25)--(2,5.25)--(3,5.75)--(3,6.25)--(1,7.25)--(1,8.25);
			
			\node[below] at (0,0) {$\ell_{n+1}$};
			\node[below] at (1,0) {$\ell_{r+1}$};
			\node[below] at (2,0) {$\ell_r$};
			\node[below] at (3,0) {$\ell_{k_s}$};
			\node[below] at (5,0) {$\ell_1$};
			
			\draw[red line, ->] (2,0)--(2,1.75)--(3,2.25)--(3,4.75)--(3.5,5)--(2.5,5.5)--(3,5.75)--(3,6.25)--(2.5,6.5)--(1,5.75)--(1,0);
			\end{tikzpicture}
%
%
%
%
			\caption{Case II-2.}\label{fcase22}
		\end{subfigure}	
			\caption{\label{figure_5_7_2} $\bullet$-new paths: Case II.}
		\end{figure}

		Similarly, one can construct new paths for (Case III) and then $\lVert {\bf i} \rVert > \lVert C_D ({\bf i}) \rVert + n$ can be established. It finishes the proof.	
	\end{proof}

\begin{example}
	We describe some examples of reduced words whose $D$- and $A$-indices are both nonzero. For each reduced word, we also provide 
	 a rigorous path neither $A$- nor $D$-canonical. 
\begin{enumerate}
	\item Let $\mathbf i = (2,1,3,2,3,1) \in \Sigma_{4}$. Then $\mathbf i \sim (2,1,\underline{3,2,1},3) \sim (2,3,\underline{1,2,3},1)$. Hence $\text{ind}_A(\mathbf i) = 1$ and $\text{ind}_D(\mathbf i )=1$, so that $\text{ind}_A(\mathbf i )  \cdot \text{ind}_D(\mathbf i ) = 1 > 0$. This decomposition $\mathbf i$ is Case I-1 in the proof of Proposition~\ref{proposition_strictly_increasing}, and the path $\ell_2 \to \ell_4 \to \ell_1 \to \ell_3$ in $G(\mathbf i,2)$ is a rigorous path which is neither $A$-canonical nor $D$-canonical (see~Figure~\ref{figure_wd_oriented}).
	\item Consider $\mathbf i = (4,3,2,1,4,2,3,2,4,3) \in \Sigma_5$. Then $\mathbf i \sim (4,3,2,4,\underline{1,2,3,4},2,3)$, so that $\text{ind}_A(\mathbf i ) = 2$ and $\text{ind}_D(\mathbf i ) = 6$. In this case wires $\ell_1$ and $\ell_5$ meet on the first column of $G(\mathbf i)$ (see~Figure~\ref{figure_D_canonical_path}). The word $\mathbf i$ is Case II-1 in the proof of Proposition~\ref{proposition_strictly_increasing}. One can see that the rigorous path $\ell_4 \to \ell_1 \to \ell_2 \to \ell_5$ is neither $A$-canonical nor $D$-canonical.
\end{enumerate}
\end{example} 

	\begin{theorem}\label{theorem_simplicial}
	For a given ${\bf i} \in \Sigma_{n+1}$, the following statements are equivalent:
	\begin{enumerate}
		\item The associated string cone $C_{\bf i}$ is simplicial.
		\item There exists a sequence $(\sigma_1,\dots,\sigma_n) \in \{A, D\}^n$ such that 
\[
\mathrm{ind}_{\sigma_k}\left( C_{\sigma_{k+1}} \circ \cdots \circ C_{\sigma_n} ({\bf i}) \right) = 0  \quad \text{ for all }k = n,\dots,1.
\]
		\item The string cone $C_{\mathbf i}$ is unimodularly equivalent to the product of cones:
		\begin{equation}\label{eq_string_cone_r}
		\begin{split}
		&t_1 \geq 0;\\
		&t_2 \geq t_3 \geq 0; \\
		& \qquad \vdots \\
		&t_{N-n+1} \geq t_{N-n+2} \geq \cdots \geq t_N \geq 0.
		\end{split}
		\end{equation}
	\end{enumerate}
\end{theorem}
\begin{proof}
	By Proposition~\ref{proposition_strictly_increasing}, we have that $(1) \iff (2)$. Since $(3) \implies (1)$, it is enough to show that $(2)\implies(3)$. Suppose that $\mathbf {i}$ satisfies $(2)$, so that there exists a sequence $(\sigma_1,\dots,\sigma_n) \in \{A,D\}^n$ such that 
	\[
	\mathbf i \sim (E_{\sigma_n}(0) \circ \cdots \circ E_{\sigma_1}(0))(\emptyset) =: \mathbf{i}'.
	\]
	Here, $\sim$ is the equivalence relation on~$\Sigma_{n+1}$ defined in~\eqref{eq_def_of_2_move}. Hence it is enough to prove that the string cone $C_{\mathbf{i}'}$ is the product of cones in~\eqref{eq_string_cone_r} by~Lemma~\ref{lemma_2_move}. 
	
	For each $k = 1,\dots,n$,  consider the word $(E_{\sigma_k}(0) \circ \cdots \circ E_{\sigma_1}(0))(\emptyset)$. Then there are  exactly $k$ many $\sigma_k$-new paths, which are $\sigma_k$-canonical paths constructed in Propositions~\ref{proposition_canonical_D} and \ref{proposition_canonical_A}. Also by Lemma~\ref{lemma_canoncial_paths_index_zero}, one can see that these paths define inequalities
	\[
	t_{\frac{k(k-1)}{2}+1} \geq t_{\frac{k(k-1)}{2}+2} \geq \cdots \geq t_{\frac{k(k-1)}{2}+k} \geq 0.
	\] 
	Hence the result follows.
\end{proof}

\section{Gelfand--Cetlin type string polytopes}
\label{secGelfandCetlinTypeStringPolytopes}

In this section, we classify string polytopes which are unimodularly equivalent to the Gelfand--Cetlin polytope. Indeed, we give a necessary and sufficient condition on a reduced word $\mathbf i \in \Sigma_{n+1}$ such that the string polytope $\Delta_{\mathbf i}(\lambda)$ is unimodularly equivalent to the Gelfand--Cetlin polytope $GC(\lambda)$ for a regular dominant weight $\lambda$ (see~Theorem~\ref{thm_GC_type_string_polytope}). 

We first recall the definition of Gelfand--Cetlin polytopes from~\cite{GC1950}.
For a dominant weight $\lambda = \lambda_1 \varpi_1 + \cdots + \lambda_n \varpi_n \in \mathbf{t}_{+}^{\ast}$, the \defi{Gelfand--Cetlin polytope}{\footnote {It is also called the \emph{Gelfand--Zeitlin polytope}, the \emph{Gelfand--Zetlin polytope}, or the \emph{Gelfand--Tsetlin polytope}.}}, or simply the \defi{GC-polytope}, denoted by $GC(\lambda)$ is a closed convex polytope in $\mathbb{R}^N$, where $N = n(n+1)/2$ as in Section~\ref{secGleizerPostnikovDescription}. Using the coordinates $\{\xkj{k}{j} \mid 1 \leq j \leq k, 1 \leq k \leq n \}$ of~$\mathbb{R}^N$,
the Gelfand--Cetlin polytope $GC(\lambda)$ consisting of points $(\xkj{k}{j})_{k,j} \in \mathbb{R}^N$ satisfying the inequalities
\[
\xkj{k+1}{j} \geq \xkj{k}{j} \geq \xkj{k+1}{j+1}, \quad 1 \leq j \leq k, 1 \leq k \leq n
\]
where $\xkj{n+1}{j} := \lambda_j+\cdots+\lambda_n$ for all $j = 1,\dots,n$ and $\xkj{n+1}{n+1} = 0$ (see~Figure~\ref{fig_GC_polytope}). 
Equivalently, a point~$(\xkj{k}{j})_{k,j}$ is contained in $GC(\lambda)$ if and only if it satisfies the following inequalities:
\begin{equation}\label{eq_GC_polytope}
\begin{tikzcd}[column sep = 0cm]
\xkj{n+1}{1} \arrow[rd, "\textcolor{black}{{\geq}}" description,sloped,	color=white]  
&& \xkj{n+1}{2} \arrow[ld, "\textcolor{black}{{\geq}}" description,sloped,	color=white]  
\arrow[rd, "\textcolor{black}{{\geq}}" description,sloped,	color=white]  
&& \xkj{n+1}{3} \arrow[ld, "\textcolor{black}{{\geq}}" description,sloped,	color=white]  
\arrow[rd, "\textcolor{black}{{\cdots}}" description,sloped,	color=white]  
& \cdots 
& \xkj{n+1}{n} \arrow[rd, "\textcolor{black}{{\geq}}" description,sloped,	color=white]  
&& \xkj{n+1}{n+1} \arrow[ld, "\textcolor{black}{{\geq}}" description,sloped,	color=white]  \\
& \xkj{n}{1} \arrow[rd, "\textcolor{black}{{\geq}}" description,sloped,	color=white]  
&& \xkj{n}{2} \arrow[rd, "\textcolor{black}{{\cdots}}" description,sloped,	color=white]  
\arrow[ld, "\textcolor{black}{{\geq}}" description,sloped,	color=white]  
&{}&{} && \xkj{n}{n} \arrow[ld, "\textcolor{black}{{\geq}}" description,sloped,	color=white]  \\
&& \xkj{n-1}{1} \arrow[rd, "\textcolor{black}{{\geq}}" description,sloped,	color=white]  &{}&{}
&& \xkj{n-1}{n-1} \arrow[ld, "\textcolor{black}{{\geq}}" description,sloped,	color=white]  \\
&&& \ddots \arrow[rd, "\textcolor{black}{{\geq}}" description,sloped,	color=white]   && 
\iddots \arrow[ld, "\textcolor{black}{{\geq}}" description,sloped,	color=white]  \\
&&&& \xkj{1}{1}
\end{tikzcd}
\end{equation}
\begin{remark}
	There are several different conventions on coordinates to express the GC-polytope $GC(\lambda)$. For instance, in the paper of Littelmann~\cite{Li}, he used coordinates $g_{i,j}$ and the coordinates for the first row of~\eqref{eq_GC_polytope} consist of $g_{1,1},g_{1,2},\dots,g_{1,n}$ (not $g_{n,1},\dots,g_{n,n}$). 
\end{remark}

\begin{figure}
	\begin{subfigure}[b]{0.49\textwidth}
		\centering
		\tdplotsetmaincoords{70}{100}
		\begin{tikzpicture}[tdplot_main_coords, scale=0.8]
		
		\coordinate (1) at (4,2,4);
		\coordinate (2) at (4,2,2);
		\coordinate (3) at (4,0,4);
		\coordinate (4) at (4,0,0);
		\coordinate (5) at (2,0,2);
		\coordinate (6) at (2,0,0);
		\coordinate (7) at (2,2,2);
		
		\begin{scope}[color=gray!50, thin]
		\foreach \xi in {0,1,2,3,4,5} { \draw (\xi, 3,0) -- (\xi, 0,0) -- (\xi, 0, 5); }%
		\foreach \yi in {0,1,2,3} {\draw (5,\yi,0) -- (0,\yi,0) -- (0,\yi,5);}%
		\foreach \zi in {0,1,2,3,4,5} {\draw (5,0,\zi) -- (0,0,\zi) -- (0,3,\zi);}%
		\end{scope}

		\draw[->] (0,0,0) -- (5.5,0,0) node[anchor = north] {$x_{2,1}$};
		\draw[->] (0,0,0) -- (0,3.5,0) node[anchor = south] {$x_{2,2}$} ;
		\draw[->] (0,0,0) -- (0,0,5.5) node[anchor = south] {$x_{1,1}$};	
		
		\draw[thick] (1)--(2)--(4)--(3)--cycle;
		\draw[thick] (1)--(7)--(2);
		\draw[thick, dashed] (7)--(6);
		\draw[thick, dashed] (4)--(6)--(5)--(3);
		\draw[thick, dashed] (7)--(5);

		\foreach \x in {1,...,7}{
			\node[circle,fill=black,inner sep=1pt] at (\x) {};
		}
		
		\end{tikzpicture}
		\caption{The GC-polytope $GC(2\varpi_1+2\varpi_2)$.}
		\label{fig_GC_polytope}
	\end{subfigure}
	\begin{subfigure}[b]{0.49\textwidth}
		\centering
\tdplotsetmaincoords{70}{100}
\begin{tikzpicture}[tdplot_main_coords, scale=0.8]

\coordinate (1) at (0,0,0);
\coordinate (2) at (2,0,0);
\coordinate (3) at (4,2,0);
\coordinate (4) at (2,4,2);
\coordinate (5) at (0,4,2);
\coordinate (6) at (0,2,2);
\coordinate (7) at (0,2,0);

\begin{scope}[color=gray!50, thin]
\foreach \xi in {0,1,2,3,4,5} { \draw (\xi, 5,0) -- (\xi, 0,0) -- (\xi, 0, 3); }%
\foreach \yi in {0,1,2,3,4,5} {\draw (5,\yi,0) -- (0,\yi,0) -- (0,\yi,3);}%
\foreach \zi in {0,1,2,3} {\draw (5,0,\zi) -- (0,0,\zi) -- (0,5,\zi);}%
\end{scope}

\draw[->] (0,0,0) -- (5.5,0,0) node[anchor = north] {$t_1$};
\draw[->] (0,0,0) -- (0,5.5,0) node[anchor = south] {$t_2$} ;
\draw[->] (0,0,0) -- (0,0,3.5) node[anchor = south] {$t_3$};	

\draw[thick] (1)--(2)--(3)--(4)--(5)--(6)--cycle;
\draw[thick] (4)--(6)--(2);
\draw[thick, dashed] (5)--(7)--(3);
\draw[thick, dashed] (7)--(1);

\foreach \x in {1,...,7}{
	\node[circle,fill=black,inner sep=1pt] at (\x) {};
}

\end{tikzpicture}
\caption{The string polytope $\Delta_{(1,2,1)}(2\varpi_1 + 2\varpi_2)$.}
\label{fig_string_polytope}
\end{subfigure}
\caption{The GC-polytope $GC(\lambda)$ and the string polytope $\Delta_{(1,2,1)}(\lambda)$ for $\lambda = 2\varpi_1 + 2\varpi_2$.}

\end{figure}
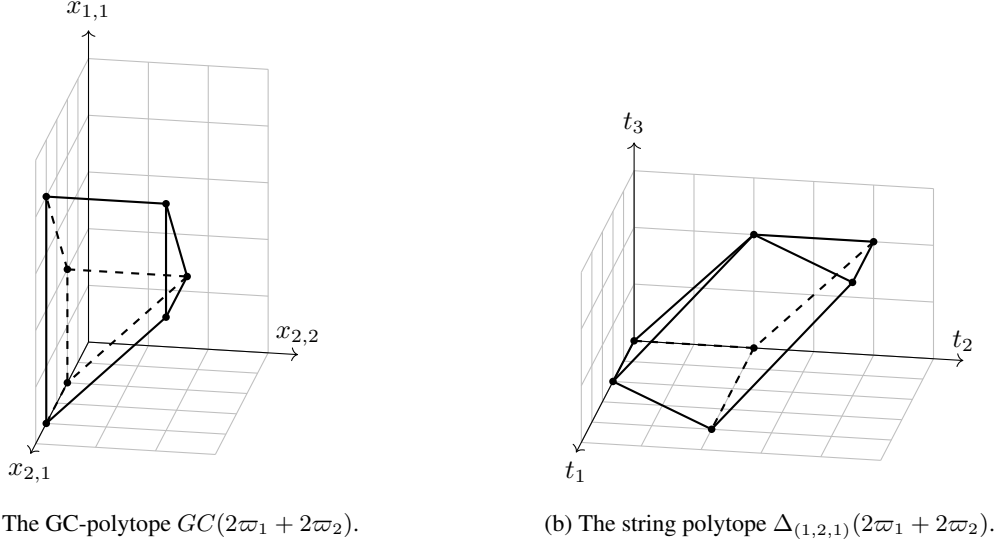

For the sequence $\sigma = (D,D,\dots,D) \in \{A,D\}^n$, we have that
\begin{equation}\label{eq_standard_word}
\mathbf i_0 :=  (\underbrace{E_{D}(0) \circ \cdots \circ E_D(0)}_{n})(\emptyset) = (1,2,1,3,2,1,\cdots,k,k-1,\dots,1,\dots,n,n-1,\dots,1).
\end{equation}
For such a word, Littelmann showed that the string polytope $\Delta_{\mathbf i_0}(\lambda)$ is unimodularly equivalent to the Gelfand--Cetlin polytope $GC(\lambda)$ for a  
dominant weight $\lambda$ in~\cite[\S5]{Li}. One can see the string polytope $\Delta_{(1,2,1)}(2\varpi_1 + 2\varpi_2)$ in Figure~\ref{fig_string_polytope}. 
On the other hand, one can see that this word defines the simplicial string cone $C_{\mathbf{i}_0}$ by Theorem~\ref{theorem_simplicial}.
In this section, we will show that, for a given reduced word $\mathbf i \in \Sigma_{n+1}$, if there exists a sequence $(\sigma_1,\dots,\sigma_n) \in \{A, D\}^n$ such that 
\[
\mathrm{ind}_{\sigma_k}\left( C_{\sigma_{k+1}} \circ \cdots \circ C_{\sigma_n} ({\bf i}) \right) = 0 
\]
for all $k = n,\dots, 1$,
then the string polytope $\Delta_{\mathbf i}(\lambda)$ is unimodularly equivalent to $GC(\lambda)$ for a dominant integral weight $\lambda$ (see Theorem~\ref{prop_i_flag_affine_equiv_GC}). 

The following lemma plays an important role to prove the unimodularity:
\begin{lemma}
\label{lemma_volumns_are_same}
	Let $\mathbf i$ be a reduced word of the longest element in $\mathfrak{S}_{n+1}$ and $\lambda$ be a dominant integral weight. Then we have that
	\[
	\textup{vol}(\Delta_{\mathbf{i}}(\lambda)) = \textup{vol}(GC(\lambda)).
	\]
	Here, $\textup{vol}$ is the Euclidean volume of a polytope. 
\end{lemma}
\begin{proof}
	It is known from~\cite[Corollary 4.2]{Kav15} that the string polytope $\Delta_{\mathbf{i}}(\lambda)$ is a Newton--Okounkov body of the full flag manifold $\mathcal{F}\ell(\C^{n+1})$ with the line bundle $\mathcal{L}_{\lambda}$ determined by $\lambda$ and the valuation determined by $\mathbf{i}$. Since the Euclidean volume of a Newton--Okounkov body encodes the degree of the projective embedding (see~\cite[Theorem 1.15]{Kav15}), for any different choices  $\mathbf{i}$ and $\mathbf{i}'$ of reduced words, we obtain
	\[
	\text{vol}(\Delta_{\mathbf{i}}(\lambda)) = \text{vol}(\Delta_{\mathbf{i}'}(\lambda)). 
	\]
	On the other hand, it was proven in~\cite[\S5]{Li} that the string polytope $\Delta_{\mathbf{i}_0}(\lambda)$ is unimodularly equivalent to the Gelfand--Cetlin polytope $GC(\lambda)$ for the word $\mathbf{i}_0$ in~\eqref{eq_standard_word}. Hence we have that for any $\mathbf{i} \in \Sigma_{n+1}$, 
	\[
	\text{vol}(\Delta_{\mathbf{i}}(\lambda)) = \text{vol}(\Delta_{\mathbf{i}_0}(\lambda)) =  \text{vol}(GC(\lambda))
	\]
	which proves the lemma.
\end{proof}

Let $\mathbf i \in \Sigma_{n+1}$ be a word. Suppose that there exists a sequence $(\sigma_1,\dots,\sigma_n) \in \{A,D\}^n$ such that 
\[
\mathrm{ind}_{\sigma_k}\left( C_{\sigma_{k+1}} \circ \cdots \circ C_{\sigma_n} ({\bf i}) \right) = 0  \quad \text{ for all }k = n,\dots,1.
\] 
Define
$
 \mathbf{i}' := (E_{\sigma_n}(0) \circ \cdots \circ E_{\sigma_1}(0))(\emptyset).
$
Then $\mathbf i \sim \mathbf i'$ by~\eqref{equation_commutation}.
Now we explicitly construct an affine transformation $T \colon \mathbf{t} \mapsto M\mathbf{t} + \mathbf{v}$ which maps $\Delta_{\mathbf i'}(\lambda)$ into $GC(\lambda)$ where $M \in \textup{GL}(N, \mathbb{Z})$ and $\mathbf{v} \in \mathbb{Z}^N$ where $N = \frac{n(n+1)}{2}$.
Since string polytopes $\Delta_{\mathbf i}$ and $\Delta_{\mathbf i'}$ are unimodularly equivalent by Lemma~\ref{lemma_2_move}, by considering the composition of $T$ and an appropriate change of coordinates, we get an affine transformation which maps $\Delta_{\mathbf i}(\lambda)$ into $GC(\lambda)$.


We may consider the matrix $M$ consisting of block matrices $M_{k,l} \in M_{k \times l}(\Z)$, and the vector $\mathbf{v}$ consisting of vectors $\mathbf{v}_k \in \Z^k$:
\[
M = \begin{bmatrix}
M_{1,1} & M_{1,2} & \cdots & M_{1,n} \\
M_{2,1} & M_{2,2} & \cdots & M_{2,n} \\
\vdots & \vdots & \ddots & \vdots \\
M_{n,1} & M_{n,2}& \cdots & M_{n,n}
\end{bmatrix}, \quad
\mathbf{v} = \begin{bmatrix}
\mathbf{v}_1 \\ \mathbf{v}_{2} \\ \vdots \\ \mathbf{v}_n
\end{bmatrix}.
\]
Let $A_k$ be the $k \times k$ matrix such that the $(i,i)$-th entry is $1$ for $i = 1,\dots,k$ and the $(i,i+1)$-th entry is $-1$ for $i = 1,\dots,k-1$ and the others are zero. 
Also, let $D_k$ be the $k \times k$ matrix such that the $(i,k-i+1)$-entry is $-1$ for $i = 1,\dots,k$ and the $(i,k-i+2)$-entry is $1$ for $i = 2,\dots,k$ and the other entries are zero. 
For instance, we have
\[
A_3 = \begin{bmatrix}
1 & -1 & 0 \\ 0 & 1 & -1 \\ 0 & 0 & 1
\end{bmatrix},\quad D_3 = \begin{bmatrix}
0 & 0 & -1 \\ 0 & -1 & 1 \\ -1 & 1 & 0
\end{bmatrix}.
\]

We inductively define matrices $M_{k,l}$ and vectors $\mathbf{v}_k$ using the sequence $\sigma$ and matrices $A_k$ and $D_k$ on $k = n,n-1,\dots,1$. 
We first let $M_{n,1}, M_{n,2},\dots,M_{n,n-1}$ be zero matrices and $M_{n,n}$ be defined by
\[
M_{n,n} = \begin{cases}
A_n & \text{ if } \sigma_n = A, \\
D_n & \text{ if } \sigma_n = D.
\end{cases}
\]
We also define the vector $\mathbf{v}_n$ by
\[
\mathbf{v}_n = \begin{cases}
(\sum_{j=2}^n \lambda_j, \sum_{j=3}^n \lambda_j, \dots,\lambda_n, 0) \in \Z^n & \text{ if } \sigma_n = A, \\
(\sum_{j=1}^n \lambda_j, \sum_{j=2}^n \lambda_j, \dots, \lambda_{n-1}+\lambda_n, \lambda_n) \in \Z^n & \text{ if } \sigma_n = D.
\end{cases}
\]	
Now we inductively define the matrices $M_{k,1}, M_{k,2},\dots,M_{k,n}$ as follows:
\[
\begin{split}
&M_{k,k} = \begin{cases}
A_k & \text{ if } \sigma_k = A, \\
D_k & \text{ if } \sigma_k = D,
\end{cases}\\
&M_{k,l} = \begin{cases}
M_{k+1,l}^A & \text{ if } \sigma_k = A, \\
M_{k+1,l}^D & \text{ if } \sigma_k = D,
\end{cases} \quad \text{ for } l \neq k.
\end{split}
\]
Here, $M_{k+1,l}^A$ (respectively, $M_{k+1,l}^D$) is the $k \times l$ matrix given by deleting the first row (respectively, the last row) of the matrix $M_{k+1,l}$. Similarly, we define $\mathbf{v}_k$ to be
\[
\mathbf{v}_k = \begin{cases}
\mathbf{v}_{k+1}^A & \text{ if } \sigma_k = A, \\
\mathbf{v}_{k+1}^D & \text{ if } \sigma_k = D
\end{cases}
\]
where $\mathbf{v}_{k+1}^A$ (respectively, $\mathbf{v}_{k+1}^D$) is the integer vector in $\Z^{k}$ given by deleting the first entry (respectively, the last entry) of the vector $\mathbf{v}_{k+1}$. 

\begin{example}\label{example_232123}
	Let $\mathbf i = (2,3,2,1,2,3) = (E_A(0) \circ E_D(0) \circ E_D(0))(\emptyset)$, i.e., $\mathbf i$ corresponds to the sequence $\sigma = (D, D, A) \in \{A,D\}^3$. Hence we have that
	\[
	\begin{split}
	& [ M_{3,1} \ M_{3,2} \ M_{3,3} ] = [ O \ O \ A_3] = \begin{bmatrix}
	0 & 0 & 0 & 1 & -1 & 0 \\  0 & 0 & 0 & 0 & 1 & -1 \\ 0 & 0 & 0 & 0 & 0 & 1
	\end{bmatrix}, \\
	& [M_{2,1} \ M_{2,2} \ M_{2,3}] = [O \ D_2 \ M_{3,3}^D] 
	= \begin{bmatrix}
	0 & 0 & -1 & 1 & - 1 & 0 \\
	0 & -1 & 1 & 0 & 1 & -1
	\end{bmatrix}, \\
	& [M_{1,1} \ M_{1,2} \ M_{1,3}] = [D_1 \ M_{2,2}^D \ M_{2,3}^D] = 
	\begin{bmatrix}
	-1 & 0 & -1 & 1 & -1 & 0
	\end{bmatrix},
	\end{split}
	\]
	and
	\[
	\mathbf{v}_3 = \begin{bmatrix}
	\lambda_2 + \lambda_3  \\ \lambda_3 \\ 0
	\end{bmatrix},\quad
	\mathbf{v}_2 = \begin{bmatrix}
	\lambda_2 + \lambda_3 \\ \lambda_3
	\end{bmatrix},\quad
	\mathbf{v}_1 = \begin{bmatrix}
	\lambda_2 + \lambda_3
	\end{bmatrix}.
	\]
\end{example}

Now we will state the following theorem of which proof will be given in Section~\ref{subsec_proof_of_i_flag_affine_equiv_GC}:
\begin{theorem}\label{prop_i_flag_affine_equiv_GC}
	Let $\mathbf i \in \Sigma_{n+1}$.
	If there exists a sequence $(\sigma_1,\dots,\sigma_n) \in \{A,D\}^n$ such that 
\begin{equation}\label{eq_ind_of_i}
\mathrm{ind}_{\sigma_k}\left( C_{\sigma_{k+1}} \circ \cdots \circ C_{\sigma_n} ({\bf i}) \right) = 0  \quad \text{ for all }k = n,\dots,1,
\end{equation}
	then the corresponding string polytope $\Delta_{\mathbf i}(\lambda)$
	is unimodularly equivalent to the Gelfand--Cetlin polytope $GC(\lambda)$ for any dominant integral weight $\lambda$. Indeed, the string polytope $\Delta_{\mathbf i}(\lambda)$ maps to the Gelfand--Cetlin polytope $GC(\lambda)$ via the composition of the affine transformation $\mathbf{t} \mapsto M \mathbf{t} + \mathbf{v}$ and some coordinate changes.
\end{theorem}
\begin{remark}
	When a given word $\mathbf i \in \Sigma_{n+1}$ satisfies the condition on Theorem~\ref{prop_i_flag_affine_equiv_GC}, there exist two different sequences $\sigma \in \{A,D\}^n$ satisfying the condition~\eqref{eq_ind_of_i}. 
	Indeed, $(1) = (E_A(0))(\emptyset) = (E_D(0))(\emptyset)$. Hence if $(A,\sigma_2,\dots,\sigma_n)$ satisfies the condition~\eqref{eq_ind_of_i}, then $(D,\sigma_2,\dots,\sigma_n)$ also satisfies~\eqref{eq_ind_of_i}, and vice versa. 
	Even though such different choices of sequences in $\{A,D\}^n$ define different affine transformations, each of affine transformations satisfies the desired property as one can see in the proof of Theorem~\ref{prop_i_flag_affine_equiv_GC}.
\end{remark}
For a parabolic subgroup $P \subset G$, we define a weight $\lambda_P$ to be the weight corresponding to the anticanonical bundle of a partial flag variety $G/P$. For example, 
we have $\lambda_{B} = 2 \varpi_1 + \cdots + 2 \varpi_n$ for a Borel subgroup $B \subset \SL_{n+1}(\C)$. 
It has been known from \cite[Theorem 7]{Ru} and \cite[\S 4]{Ste19} that $\Delta_{\bf i}(\lambda_P)$ contains exactly one lattice point $\iota$ in its interior and the dual polytope $(\Delta_{\bf i}(\lambda_P)-\iota)^*$ is integral for  any ${\bf i} \in \Sigma_{n+1}$. 
Hence together with Theorem~\ref{prop_i_flag_affine_equiv_GC}, we get the following corollary:
\begin{corollary}
	Let $\mathbf i \in \Sigma_{n+1}$.
	If there exists a sequence $(\sigma_1,\dots,\sigma_n) \in \{A,D\}^n$ such that 
\[
\mathrm{ind}_{\sigma_k}\left( C_{\sigma_{k+1}} \circ \cdots \circ C_{\sigma_n} ({\bf i}) \right) = 0  \quad \text{ for all }k = n,\dots,1,
\]
	then the corresponding string polytope $\Delta_{\mathbf i}(\lambda)$ is an integral polytope, i.e., its vertices are contained in $\Z^N$ for any dominant integral weight $\lambda$.
	Moreover, the string polytope $\Delta_{\bf i}(\lambda_P)$ for a partial flag variety $G/P$ is reflexive (after translation by a lattice vector), i.e., the dual $(\Delta_{\bf i}(\lambda_P)-\iota)^{\ast}$ is also an integral polytope. 
\end{corollary}

To sum up, we may summarize our main results as follows. 

\begin{theorem}\label{thm_GC_type_string_polytope}
	Let $\mathbf i$ be a reduced word of the longest element in the Weyl group of $\SL_{n+1}(\C)$, and let $\lambda$ be a regular dominant integral weight. Then the following are equivalent:
	\begin{enumerate}
		\item The string polytope $\Delta_{\bf i}(\lambda)$ is unimodularly equivalent to the Gelfand--Cetlin polytope $GC(\lambda)$. 
		\item The string polytope $\Delta_{\bf i}(\lambda)$ has exactly $n(n+1)$ many facets.
		\item The associated string cone $C_{\bf i}$ is simplicial.
		\item There exists a sequence $(\sigma_1,\dots,\sigma_n) \in \{A,D\}^n$ such that 
\[
\mathrm{ind}_{\sigma_k}\left( C_{\sigma_{k+1}} \circ \cdots \circ C_{\sigma_n} ({\bf i}) \right) = 0  \quad \text{ for all }k = n,\dots,1.
\]
	\end{enumerate}
\end{theorem}
\begin{proof}
	We already proved $(3) \iff (4)$ in Theorem~\ref{theorem_simplicial}. 
	Also, by Theorem~\ref{prop_i_flag_affine_equiv_GC}, we have $(4) \implies (1)$. 	
	Since the number of facets of the Gelfand--Cetlin polytope $GC(\lambda)$ is $n(n+1)$, we have $(1) \implies (2)$. Finally, since the defining inequalities in Theorem~\ref{theorem_Li} has no redundancy (see~Section~\ref{secNonRedundancyOfStringInequalities}), if the string polytope $\Delta_{\mathbf{i}}(\lambda)$ has exactly $n(n+1)$ many facets, then the string cone $C_{\mathbf{i}}$ has $n(n+1)/2$ many facets, i.e., it is simplicial. This proves $(2) \implies (3)$, which completes the proof.
\end{proof}

We would like to enclose this section presenting several questions which arise naturally.
Using a computer program, such as SAGE, one can check that string polytopes $\Delta_{\bf i}(\lambda_P)$ 
is integral for any parabolic subgroup $P \subset \SL_{n+1}(\C)$ when $n \leq 4$. However, when $n = 5$, Steinert~\cite[Example 7.5]{Ste19} provided the string polytope $\Delta_{\mathbf i}(\varpi_3)$ which is not integral for $\mathbf i = (1,3,2,1,3,2,4,3,2,1,5,4,3,2,1)$. But one can check that the string polytope $\Delta_{\mathbf i}(2\varpi_3)$ is integral using SAGE. Moreover, the string polytope $\Delta_{\mathbf i }(6\varpi_3)$ is integral and $6\varpi_3 = \lambda_P$, where $P$ is the maximal parabolic subgroup of $\SL_6(\C)$ satisfying that $\SL_6(\C)/P = \Gr(3,6)$. 
On the other hand, one can also check 
that the string polytope $\Delta_{(1,3,2,1,3,2,4,3,2,1,5,4,3,2,1,6,5,4,3,2,1)}(7\varpi_3)$ is not an integral polytope, where $7 \varpi_3$ is the weight corresponding to the anticanonical bundle of $\SL_7(\C)/P_3= \Gr(3,7)$. 
Hence we present the following question:
\begin{question_num}
	Can we find conditions on $\mathbf i \in \Sigma_{n+1}$ and $P \subset \SL_{n+1}(\C)$
	such that the string polytope $\Delta_{\bf i}(\lambda_P) $ is integral?
\end{question_num}


\section{Proof of Theorem~\ref{prop_i_flag_affine_equiv_GC}}
\label{subsec_proof_of_i_flag_affine_equiv_GC}

In this section, we give the proof of Theorem~\ref{prop_i_flag_affine_equiv_GC}. 
Given a reduced word $\mathbf i \in \Sigma_{n+1}$, suppose that there is a sequence $(\sigma_1,\dots,\sigma_n) \in \{A,D\}^n$  such that 
\[
\mathrm{ind}_{\sigma_k}\left( C_{\sigma_{k+1}} \circ \cdots \circ C_{\sigma_n} ({\bf i}) \right) = 0  \quad \text{ for all }k = n,\dots,1.
\]
Define
\[
\mathbf{i}' := (E_{\sigma_n}(0) \circ \cdots \circ E_{\sigma_1}(0))(\emptyset) \quad \text{so that $\mathbf i \sim \mathbf i'$ by~\eqref{equation_commutation}.}
\]
Since string polytopes $\Delta_{\mathbf i}$ and $\Delta_{\mathbf i'}$ are unimodularly equivalent by Lemma~\ref{lemma_2_move},
we may assume that $\mathbf i = \mathbf i'$. 

Before giving the proof, we need to change the coordinates for notational simplicity. Let
\[
(\tkj{1}{1}, \tkj{2}{1}, \tkj{2}{2},\dots,\tkj{n}{1},\dots,\tkj{n}{n}) \in \R^1 \times \R^2 \times \cdots \times \R^n = \R^N
\]
be the coordinate for the string polytope $\Delta_{\mathbf{i}}(\lambda)$, and let
\[
\mathbf i =: (i_{1,1}, i_{2,1},i_{2,2},\dots,i_{n,1},\dots,i_{n,n}). 
\]
Also we denote the defining function of the $\lambda$-inequality associated to $\tkj{k}{j}$ by $\skj{k}{j}$ (see Definition~\ref{definition_lambda_inequality}). Indeed, $\skj{k}{j}$ is a linear functional on the variables $(\tkj{k}{j})_{k,j}$ whose leading term appears on the variable $\tkj{k}{j}$ such that the $\lambda$-inequality is given by
\begin{equation}\label{eq_string_skj}
\skj{k}{j} + \lambda_{i_{k,j}}\geq 0.
\end{equation}
Here $\lambda_{i_{k,j}} = \langle \lambda, \alpha_{i_{k,j}}^{\vee} \rangle$. 
\begin{example}
	Suppose that $\mathbf i = (1,2,1) \in \Sigma_3$.
	Let $\lambda = \lambda_1 \varpi_1 + \lambda_2 \varpi_2$ be a 
dominant weight. Using coordinates $\tkj{k}{j}$, the string polytope $\Delta_{\mathbf i}(\lambda)$ is defined by the following inequalities:
	\[
	\begin{split}
	\tkj{1}{1} \geq 0;& \\
	\tkj{2}{1} \geq \tkj{2}{2} \geq 0;& \\
	\skj{1}{1} +\lambda_1 = -\tkj{1}{1} + \tkj{2}{1} - 2 \tkj{2}{2} + \lambda_1 \geq 0;& \\
	\skj{2}{1} +\lambda_2 = -\tkj{2}{1} + \tkj{2}{2} + \lambda_2 \geq 0;& \\
	\skj{2}{2} +\lambda_1 = -\tkj{2}{2} + \lambda_1 \geq 0.&
		\end{split}
	\]
\end{example}

For $k = 1,\dots,n-1$, we observe the relation on values $i_{k,j}$ and $i_{k+1,j}$. There are four cases on the pair~$(\sigma_k, \sigma_{k+1})$.
\begin{enumerate}
	\item[\textbf{\underline{Case 1.}}] $(\sigma_k, \sigma_{k+1}) = (A,A)$. 
	\begin{equation}\label{eq_i_kj_AA}
	\begin{tikzcd}[column sep = 0.2cm, row sep = 0.2cm]
	i_{k+1,1} \arrow[r,  "\textcolor{black}{{<}}" description,sloped,	color=white] 
	& i_{k+1,2} \arrow[r,  "\textcolor{black}{{<}}" description,sloped,	color=white] \arrow[d, equal] 
	& \cdots \arrow[r,  "\textcolor{black}{{<}}" description,sloped,	color=white] 
	& i_{k+1,j+1} \arrow[r,  "\textcolor{black}{{<}}" description,sloped,	color=white] \arrow[d, equal]
	& \cdots \arrow[r,  "\textcolor{black}{{<}}" description,sloped,	color=white] 
	& i_{k+1,k} \arrow[r,  "\textcolor{black}{{<}}" description,sloped,	color=white] \arrow[d, equal]
	& i_{k+1,k+1} \arrow[d, equal]\\
	& i_{k,1} \arrow[r,  "\textcolor{black}{{<}}" description,sloped,	color=white] & \cdots \arrow[r,  "\textcolor{black}{{<}}" description,sloped,	color=white] & i_{k,j} \arrow[r,  "\textcolor{black}{{<}}" description,sloped,	color=white] & \cdots \arrow[r,  "\textcolor{black}{{<}}" description,sloped,	color=white] & i_{k,k-1} \arrow[r,  "\textcolor{black}{{<}}" description,sloped,	color=white] & i_{k,k}
	\end{tikzcd}
	\end{equation}
	\item[\textbf{\underline{Case 2.}}] $(\sigma_k, \sigma_{k+1}) = (D,A)$.
	\begin{equation}\label{eq_i_kj_DA}
	\begin{tikzcd}[column sep = 0.2cm, row sep = 0.2cm]
	i_{k+1,1} \arrow[r,  "\textcolor{black}{{<}}" description,sloped,	color=white] 
	& i_{k+1,2} \arrow[r,  "\textcolor{black}{{<}}" description,sloped,	color=white] \arrow[d, equal] 
	& \cdots \arrow[r,  "\textcolor{black}{{<}}" description,sloped,	color=white] 
	& i_{k+1,k-j+2} \arrow[r,  "\textcolor{black}{{<}}" description,sloped,	color=white] \arrow[d, equal]
	& \cdots \arrow[r,  "\textcolor{black}{{<}}" description,sloped,	color=white] 
	& i_{k+1,k} \arrow[r,  "\textcolor{black}{{<}}" description,sloped,	color=white] \arrow[d, equal]
	& i_{k+1,k+1} \arrow[d, equal]\\
	& i_{k,k} \arrow[r,  "\textcolor{black}{{<}}" description,sloped,	color=white] & \cdots \arrow[r,  "\textcolor{black}{{<}}" description,sloped,	color=white] & i_{k,j} \arrow[r,  "\textcolor{black}{{<}}" description,sloped,	color=white] & \cdots \arrow[r,  "\textcolor{black}{{<}}" description,sloped,	color=white] & i_{k,2} \arrow[r,  "\textcolor{black}{{<}}" description,sloped,	color=white] & i_{k,1}
	\end{tikzcd}
	\end{equation}
	\item[\textbf{\underline{Case 3.}}] $(\sigma_k, \sigma_{k+1}) = (A,D)$. 
	\begin{equation}\label{eq_i_kj_AD}
	\begin{tikzcd}[column sep = 0.2cm, row sep = 0.2cm]
	i_{k+1,1} \arrow[r,  "\textcolor{black}{{>}}" description,sloped,	color=white] 
	& i_{k+1,2} \arrow[r,  "\textcolor{black}{{>}}" description,sloped,	color=white] \arrow[d, equal] 
	& \cdots \arrow[r,  "\textcolor{black}{{>}}" description,sloped,	color=white] 
	& i_{k+1,k-j+2} \arrow[r,  "\textcolor{black}{{>}}" description,sloped,	color=white] \arrow[d, equal]
	& \cdots \arrow[r,  "\textcolor{black}{{>}}" description,sloped,	color=white] 
	& i_{k+1,k} \arrow[r,  "\textcolor{black}{{>}}" description,sloped,	color=white] \arrow[d, equal]
	& i_{k+1,k+1} \arrow[d, equal]\\
	& i_{k,k} \arrow[r,  "\textcolor{black}{{>}}" description,sloped,	color=white] & \cdots \arrow[r,  "\textcolor{black}{{>}}" description,sloped,	color=white] & i_{k,j} \arrow[r,  "\textcolor{black}{{>}}" description,sloped,	color=white] & \cdots \arrow[r,  "\textcolor{black}{{>}}" description,sloped,	color=white] & i_{k,2} \arrow[r,  "\textcolor{black}{{>}}" description,sloped,	color=white] & i_{k,1}
	\end{tikzcd}
	\end{equation}
	\item[\textbf{\underline{Case 4.}}] $(\sigma_k, \sigma_{k+1}) = (D,D)$. 
	\begin{equation}\label{eq_i_kj_DD}
	\begin{tikzcd}[column sep = 0.2cm, row sep = 0.2cm]
	i_{k+1,1} \arrow[r,  "\textcolor{black}{{>}}" description,sloped,	color=white] 
	& i_{k+1,2} \arrow[r,  "\textcolor{black}{{>}}" description,sloped,	color=white] \arrow[d, equal] 
	& \cdots \arrow[r,  "\textcolor{black}{{>}}" description,sloped,	color=white] 
	& i_{k+1,j+1} \arrow[r,  "\textcolor{black}{{>}}" description,sloped,	color=white] \arrow[d, equal]
	& \cdots \arrow[r,  "\textcolor{black}{{>}}" description,sloped,	color=white] 
	& i_{k+1,k} \arrow[r,  "\textcolor{black}{{>}}" description,sloped,	color=white] \arrow[d, equal]
	& i_{k+1,k+1} \arrow[d, equal]\\
	& i_{k,1} \arrow[r,  "\textcolor{black}{{>}}" description,sloped,	color=white] & \cdots \arrow[r,  "\textcolor{black}{{>}}" description,sloped,	color=white] & i_{k,j} \arrow[r,  "\textcolor{black}{{>}}" description,sloped,	color=white] & \cdots \arrow[r,  "\textcolor{black}{{>}}" description,sloped,	color=white] & i_{k,k-1} \arrow[r,  "\textcolor{black}{{>}}" description,sloped,	color=white] & i_{k,k}
	\end{tikzcd}
	\end{equation}
\end{enumerate}

Now we present three lemmas which play important roles in the proof.

\begin{lemma}\label{lemma_proof_Proposition_GC_3}
	Let $\lambda_{k,j} = \mathbf{v}_k(j) - \mathbf{v}_k(j+1)$ for $\mathbf{v}_k = (\mathbf{v}_k(1),\dots,\mathbf{v}_k(k)) \in \Z^k$ for $1 \leq k \leq n$ and $1 \leq j \leq k-1$. Then we have that
	\[
	\lambda_{k,j} = \begin{cases}
	\lambda_{i_{k,j+1}} & \text{ if } \sigma_k = A, \\
	\lambda_{i_{k,k-j+1}} & \text{ if } \sigma_k = D.
	\end{cases}
	\]
\end{lemma}
\begin{proof}
	We use an induction argument on $k = n,n-1,\dots,1$. Suppose that $k = n$ and $\sigma_n = A$. Then $i_{n,j} = j$. On the other hand, $\lambda_{n,j} = \mathbf{v}_n(j) - \mathbf{v}_n(j+1) = \lambda_{j+1} = \lambda_{i_{n,j+1}}$. Hence the lemma holds. Consider the case when $k = n$ and $\sigma_n = D$. Then $i_{n,j} = n-j+1$, also we have that $\lambda_{n,j} = \lambda_j = \lambda_{i_{n,n-j+1}}$, so that the lemma holds when $k = n$.
	
	Now we assume that the lemma holds for $k+1$. We first note that 
	\begin{equation}\label{eq_L_kj}
	\lambda_{k,j} = \begin{cases}
	\lambda_{k+1,j+1} & \text{ if } \sigma_k = A, \\
	\lambda_{k+1,j} & \text{ if } \sigma_k = D
	\end{cases}
	\end{equation}
	by definition of the vector $\mathbf{v}_k$ and the induction hypothesis. 	
	Then we can prove the lemma using the case by case analysis on the pair $(\sigma_k, \sigma_{k+1})$.
	\begin{enumerate}
	\item $(\sigma_k, \sigma_{k+1}) = (A,A)$: By~\eqref{eq_i_kj_AA} and \eqref{eq_L_kj}, we get $\lambda_{k,j} = \lambda_{k+1,j+1} = \lambda_{i_{k+1,j+2}} = \lambda_{i_{k,j+1}}$.
	\item $(\sigma_k, \sigma_{k+1}) = (D,A)$: By~\eqref{eq_i_kj_DA} and \eqref{eq_L_kj}, we get $\lambda_{k,j} = \lambda_{k+1,j} = \lambda_{i_{k+1,j+1}} = \lambda_{i_{k,k-j+1}}$.
	\item $(\sigma_k, \sigma_{k+1}) = (A,D)$: By~\eqref{eq_i_kj_AD} and \eqref{eq_L_kj}, we get $\lambda_{k,j} = \lambda_{k+1,j+1} = \lambda_{i_{k+1,k-j+1}} = \lambda_{i_{k,j+1}}$.
	\item $(\sigma_k, \sigma_{k+1}) = (D,D)$: By~\eqref{eq_i_kj_DD} and \eqref{eq_L_kj}, we get $\lambda_{k,j} = \lambda_{k+1,j} = \lambda_{i_{k+1, k-j+2}} = \lambda_{i_{k,k-j+1}}$.
	\end{enumerate}
	Hence the result follows.
\end{proof}

\begin{lemma}\label{lemma_proof_Proposition_GC_1}
	For $k \in [n]$ and $j \in [k]$, the linear functional $\skj{k}{j}$ satisfies the following equality:
	\[
	\skj{k}{j} = \begin{cases}
	-\tkj{k}{j} + \tkj{k}{j+1} + \tkj{k+1}{j} - \tkj{k+1}{j+1} + \skj{k+1}{j+1} & \text{ if } (\sigma_k,\sigma_{k+1}) = (A,A) \text{ or }(D,D), \\
	-\tkj{k}{j} + \tkj{k}{j+1} + \tkj{k+1}{k-j+1} - \tkj{k+1}{k-j+2} + \skj{k+1}{k-j+2} & \text{ if }(\sigma_k, \sigma_{k+1}) = (D,A) \text{ or }(A,D).
	\end{cases}
	\]
	Here, we put $\tkj{k}{k+1} = 0$ and $\tkj{n+1}{j} = 0$.
\end{lemma}
\begin{proof}
	We prove this lemma by an induction argument on $k = n, n-1,\dots,1$. Suppose that $k = n$. Then we have that
	\[
	\skj{n}{j} = - \tkj{n}{j} + \tkj{n}{j+1}
	\]
	for both of cases $\sigma_n = A$ and $\sigma_n = D$. Hence the lemma follows. 

	Now we suppose that $k < n$. Then we can prove the lemma from the case by case analysis on the pair $(\sigma_k, \sigma_{k+1})$:
	\begin{enumerate}
		\item $(\sigma_k, \sigma_{k+1}) = (A,A)$ or $(D,D)$: By~\eqref{eq_i_kj_AA} and~\eqref{eq_i_kj_DD}, 
		\[
		\skj{k}{j} = -\tkj{k}{j} + \tkj{k}{j+1} + \tkj{k+1}{j} - \tkj{k+1}{j+1} + \skj{k+1}{j+1}.
		\]
		\item $(\sigma_k, \sigma_{k+1}) = (D,A)$ or $(A,D)$: By~\eqref{eq_i_kj_DA} and~\eqref{eq_i_kj_AD}, 
		\[
		\skj{k}{j} = -\tkj{k}{j} + \tkj{k}{j+1} + \tkj{k+1}{k-j+1} - \tkj{k+1}{k-j+2} + \skj{k+1}{k-j+2}. \qedhere
		\] 
	\end{enumerate}
\end{proof}

In the setting of Theorem~\ref{prop_i_flag_affine_equiv_GC}, the affine transformation $T \colon \mathbf{t} \mapsto M \mathbf{t} + \mathbf{v}$ defined by the matrix $M$ and the vector $\mathbf{v}$ in Section \ref{secGelfandCetlinTypeStringPolytopes}
maps the polytope $\Delta_{\mathbf i}(\lambda)$ into $GC(\lambda)$ using the coordinates $(\xkj{1}{1}, \xkj{2}{1}, \xkj{2}{2},\dots,\xkj{n}{1},\dots,\xkj{n}{n})$ for the codomain. 

\begin{lemma}\label{lemma_proof_Proposition_GC_2}
	For $k \in [n] $ and $j \in [k-1]$, $\xkj{k}{j}$ satisfies the following equality:
	\[
	\xkj{k}{j} - \xkj{k}{j+1} = \begin{cases}
	\tkj{k}{j} - \tkj{k}{j+1} + \skj{k}{j+1} + \lambda_{k+1,j+1} & \text{ if } \sigma_k = A, \\
	\tkj{k}{k-j} - \tkj{k}{k-j+1} + \skj{k}{k-j+1} + \lambda_{k+1,j} & \text{ if } \sigma_k = D.
	\end{cases}
	\]
\end{lemma}
\begin{proof}
	We use an induction argument on $k = n,n-1,\dots,1$ and the case by case analysis on the pair $(\sigma_k, \sigma_{k+1})$ to prove this lemma. First, we notice that by the definition of $\xkj{k}{j}$, we have that
	\begin{equation}\label{eq_xkj_AD}
	\xkj{k}{j} = \begin{cases}
	\xkj{k+1}{j+1} + \tkj{k}{j} - \tkj{k}{j+1} & \text{ if } \sigma_k = A, \\
	\xkj{k+1}{j} - \tkj{k}{k-j+1} + \tkj{k}{k-j+2} & \text{ if } \sigma_k = D.
	\end{cases}
	\end{equation}
 We may put $\lambda_{n+1,j} = \xkj{n+1}{j}- \xkj{n+1}{j+1}$ for $j = 1,\dots,n$. 
	
	Consider the case when $k = n$. We know that 
	\[
	\xkj{n}{j} =  \begin{cases}
	\tkj{n}{j} - \tkj{n}{j+1} + \xkj{n+1}{j+1} & \text{ if } \sigma_n = A, \\
	- \tkj{n}{n-j+1} + \tkj{n}{n-j+2} + \xkj{n+1}{j} & \text{ if } \sigma_n = D.
	\end{cases}
	\]
	by~\eqref{eq_xkj_AD}. 
		Hence when $\sigma_n = A$, we get that
	\[
	\xkj{n}{j} - \xkj{n}{j+1} = \tkj{n}{j} - \tkj{n}{j+1} - (\tkj{n}{j+1} - \tkj{n}{j+2}) + (\xkj{n+1}{j+1} - \xkj{n+1}{j+2})= \tkj{n}{j} - \tkj{n}{j+1} + \skj{n}{j+1} + \lambda_{n+1,j+1},
	\]
	which shows the case when $\sigma_n = A$. When $\sigma_n = D$, we have
	\[
	\begin{split}
	\xkj{n}{j} - \xkj{n}{j+1} &= - \tkj{n}{n-j+1} + \tkj{n}{n-j+2} - (-\tkj{n}{n-j} + \tkj{n}{n-j+1}) + \xkj{n+1}{j} - \xkj{n+1}{j+1} \\
	&= \tkj{n}{n-j} - \tkj{n}{n-j+1} + \skj{n}{n-j+1} + \lambda_{n+1,j},
	\end{split}
	\]
	which proves the case when $k = n$.
	
	Suppose that the lemma holds for $k+1,\dots,n$. If $\sigma_k = A$, then we have 
	\begin{equation}\label{eq_xkj_lemma_1}
	\xkj{k}{j} - \xkj{k}{j+1} = \xkj{k+1}{j+1} + \tkj{k}{j} - \tkj{k}{j+1} - \xkj{k+1}{j+2} - \tkj{k}{j+1} + \tkj{k}{j+2}
	\end{equation}
	by~\eqref{eq_xkj_AD}. When $\sigma_{k+1} = A$, using the induction hypothesis, the above equality~\eqref{eq_xkj_lemma_1} becomes to
	\[
	\begin{split}
	\xkj{k}{j} - \xkj{k}{j+1}
	&= \tkj{k+1}{j+1} - \tkj{k+1}{j+2} + \skj{k+1}{j+2} + \lambda_{k+2,j+2} + \tkj{k}{j} - \tkj{k}{j+1} - \tkj{k}{j+1} + \tkj{k}{j+2} \\
	&= \tkj{k}{j} - \tkj{k}{j+1} + \skj{k}{j+1} + \lambda_{k+1,j+1}.
	\end{split}
	\]
	The second equality comes from Lemma~\ref{lemma_proof_Proposition_GC_1} and~\eqref{eq_L_kj}. When $\sigma_{k+1}=D$, then again by the induction hypothesis, the equality~\eqref{eq_xkj_lemma_1} changes to
	\[
	\begin{split}
	\xkj{k}{j} - \xkj{k}{j+1} 
	&= \tkj{k+1}{k-j} - \tkj{k+1}{k-j+1} + \skj{k+1}{k-j+1} + \lambda_{k+2,j+1} + \tkj{k}{j} - \tkj{k}{j+1} - \tkj{k}{j+1} + \tkj{k}{j+2} \\
	&= \tkj{k}{j} - \tkj{k}{j+1} + \skj{k}{j+1} + \lambda_{k+1,j+1}.
	\end{split}
	\]
		The second equality comes from Lemma~\ref{lemma_proof_Proposition_GC_1} and~\eqref{eq_L_kj}. 
		
	If $\sigma_k = D$, then we get
	\begin{equation}\label{eq_xkj_lemma_2}
	\xkj{k}{j}- \xkj{k}{j+1} = \xkj{k+1}{j} -\tkj{k}{k-j+1} + \tkj{k}{k-j+2} - \xkj{k+1}{j+1} + \tkj{k}{k-j} - \tkj{k}{k-j+1}
	\end{equation}
	by~\eqref{eq_xkj_AD}. Consider the case when $\sigma_{k+1} = A$. Then using the induction hypothesis, the equality~\eqref{eq_xkj_lemma_2} becomes
	\[
	\begin{split}
	\xkj{k}{j} - \xkj{k}{j+1} &= \tkj{k+1}{j} - \tkj{k+1}{j+1} + \skj{k+1}{j+1} + \lambda_{k+2,j+1} - \tkj{k}{k-j+1} + \tkj{k}{k-j+2} + \tkj{k}{k-j} - \tkj{k}{k-j+1} \\
	&= \tkj{k}{k-j} - \tkj{k}{k-j+1} + \skj{k}{k-j+1} + \lambda_{k+1,j}.
	\end{split}
	\]
			The second equality comes from Lemma~\ref{lemma_proof_Proposition_GC_1} and~\eqref{eq_L_kj}. 
	Finally, when $\sigma_{k+1} = D$, then by the induction, the equality~\eqref{eq_xkj_lemma_2} changes to
	\[
	\begin{split}
	\xkj{k}{j} - \xkj{k}{j+1} &= \tkj{k+1}{k-j+1} - \tkj{k+1}{k-j+2} + \skj{k+1}{k-j+2} + \lambda_{k+2,j} - \tkj{k}{k-j+1} + \tkj{k}{k-j+2} + \tkj{k}{k-j} - \tkj{k}{k-j+1} \\
	&= \tkj{k}{k-j} - \tkj{k}{k-j+1} + \skj{k}{k-j+1} + \lambda_{k+1,j}.
	\end{split}
	\]
	The second equality comes from Lemma~\ref{lemma_proof_Proposition_GC_1} and~\eqref{eq_L_kj}. Hence the result follows.
\end{proof}

Before giving the proof of Theorem~\ref{prop_i_flag_affine_equiv_GC}, let us look at a simple example.

\begin{example}
	Let $\mathbf i = (2,3,2,1,2,3)$. Then by Example~\ref{example_232123}, one can see that the affine transformation in~Theorem~\ref{prop_i_flag_affine_equiv_GC} maps $\mathbf{t} = (\tkj{1}{1}, \tkj{2}{1}, \tkj{2}{2}, \tkj{3}{1},\tkj{3}{2},\tkj{3}{3}) \in\Delta_{\bf i}(\lambda)$  to
	\[
	\begin{tikzcd}[column sep = -0.5cm, row sep = -0cm]
	\xkj{3}{1} && \xkj{3}{2} && \xkj{3}{3} \\
	& \xkj{2}{1} && \xkj{2}{2} \\
	&& \xkj{1}{1}
	\end{tikzcd} = 
	\begin{tikzcd}[column sep = -2cm, row sep = 0cm]
	\tkj{3}{1} - \tkj{3}{2} + \lambda_2 + \lambda_3 && \tkj{3}{2} - \tkj{3}{3} + \lambda_3 && \tkj{3}{3} \\
	& -\tkj{2}{2} + \tkj{3}{1} - \tkj{3}{2} + \lambda_2 +\lambda_3 &&  -\tkj{2}{1} + \tkj{2}{2} + \tkj{3}{2} - \tkj{3}{3} + \lambda_3 \\
	&& -\tkj{1}{1} - \tkj{2}{2} + \tkj{3}{1} - \tkj{3}{2} + \lambda_2 + \lambda_3
	\end{tikzcd}
	\]
	We will show that this point is contained in the Gelfand--Cetlin polytope, i.e., it satisfies the relation 
	\begin{equation}\label{eq_relation_skj}
	\xkj{k+1}{j} \geq \xkj{k}{j} \geq \xkj{k+1}{j+1} \quad \text{for $1 \leq j \leq k$ and $1 \leq k \leq 3$. }
	\end{equation}	
	
	When $k = 3$, then since we have $\tkj{3}{1} \geq \tkj{3}{2} \geq \tkj{3}{3} \geq 0$ by Theorem~\ref{theorem_simplicial}, one can see
	\[
	\tkj{3}{1} - \tkj{3}{2} + \lambda_2 + \lambda_3  \geq \lambda_2 + \lambda_3, \quad 
	\tkj{3}{2} - \tkj{3}{3} + \lambda_3 \geq \lambda_3, \quad
	 \tkj{3}{3} \geq 0.
	\]
	One the other hand, from $\lambda$-inequalities we get
	\[
	\begin{split}
	& \lambda_1 + \lambda_2 + \lambda_3 - (\tkj{3}{1} - \tkj{3}{2} + \lambda_2 + \lambda_3)
	= \lambda_1 - \tkj{3}{1} + \tkj{3}{2} = \skj{3}{1} + \lambda_1 \geq 0; \\
	& \lambda_2 + \lambda_3 - (\tkj{3}{2} - \tkj{3}{3} + \lambda_3)
	= \lambda_2  - \tkj{3}{2} + \tkj{3}{3} = \skj{3}{2} + \lambda_2 \geq 0; \\
	& \lambda_3 - \tkj{3}{3} = \skj{3}{3} + \lambda_3 \geq 0.
	\end{split}
	\]
	Hence the inequalities in~\eqref{eq_relation_skj} hold when $k = 3$. 
	
	Consider the case when $k = 2$. Since $\tkj{2}{1} \geq \tkj{2}{2} \geq 0$, we get 
	\[
	\tkj{3}{1} - \tkj{3}{2} + \lambda_2 + \lambda_3 \geq  -\tkj{2}{2} + \tkj{3}{1} - \tkj{3}{2} + \lambda_2 +\lambda_3, \quad
	 \tkj{3}{2} - \tkj{3}{3} + \lambda_3  \geq -\tkj{2}{1} +\tkj{2}{2}+ \tkj{3}{2} - \tkj{3}{3} + \lambda_3.
	\]
	Moreover, we have 
	\[
	\begin{split}
	&-\tkj{2}{2} + \tkj{3}{1} - \tkj{3}{2} + \lambda_2 +\lambda_3 -( \tkj{3}{2} - \tkj{3}{3} + \lambda_3) = \skj{2}{2} + \lambda_2 \geq 0, \\
	& -\tkj{2}{1} + \tkj{2}{2} + \tkj{3}{2} - \tkj{3}{3} + \lambda_3 - \tkj{3}{3} = \skj{2}{1} + \lambda_3 \geq 0.
	\end{split}
	\]
	Therefore the inequalities~\eqref{eq_relation_skj} also hold when $k = 2$.
	
	Finally, when $k = 1$, the inequality $\tkj{1}{1} \geq 0$ implies that
	\[
	-\tkj{2}{2} + \tkj{3}{1} - \tkj{3}{2} + \lambda_2 +\lambda_3 \geq -\tkj{1}{1} -\tkj{2}{2} + \tkj{3}{1} - \tkj{3}{2} + \lambda_2 +\lambda_3.
	\]
	Also, we have that
	\[
	-\tkj{1}{1} -\tkj{2}{2} + \tkj{3}{1} - \tkj{3}{2} + \lambda_2 +\lambda_3 - (-\tkj{2}{1} +\tkj{2}{2}+ \tkj{3}{2} - \tkj{3}{3} + \lambda_3 )
	=\skj{1}{1} + \lambda_2 \geq 0.
	\]
	This shows that the affine transformation in Theorem~\ref{prop_i_flag_affine_equiv_GC} maps the string polytope to the GC-polytope. \qed
\end{example}

We finally give the proof of Theorem \ref{prop_i_flag_affine_equiv_GC}.

\begin{proof}[Proof of Theorem~\ref{prop_i_flag_affine_equiv_GC}] 
Since the volumes of two polytopes $\Delta_{\mathbf i}(\lambda)$ and $GC(\lambda)$ are same by~Lemma~\ref{lemma_volumns_are_same}, it is enough to show that the affine transformation $T$  defined by the matrix $M$ and the vector $\mathbf{v}$ maps the polytope $\Delta_{\mathbf i}(\lambda)$ into $GC(\lambda)$ using the coordinates $(\xkj{1}{1}, \xkj{2}{1}, \xkj{2}{2},\dots,\xkj{n}{1},\dots,\xkj{n}{n})$ for the codomain.  Indeed, it is enough to show that
\begin{equation}\label{eq_GC_inequl_xkj}
\xkj{k+1}{j} \geq \xkj{k}{j} \geq \xkj{k+1}{j+1} 
\end{equation}
for $j = 1,\dots, k$ and $k = 1,\dots,n$. 

We first consider the case when $\sigma_k = A$. Since, by Theorem~\ref{theorem_simplicial}, the string cone inequality is given by 
\[
\tkj{k}{1} \geq \tkj{k}{2} \geq \cdots \geq \tkj{k}{k} \geq 0,
\]
we have that $\xkj{k}{j} \geq \xkj{k+1}{j+1}$ by~\eqref{eq_xkj_AD}. Considering $\xkj{k+1}{j} - \xkj{k}{j}$, we have that
\[
\xkj{k+1}{j} - \xkj{k}{j} = \xkj{k+1}{j} - \xkj{k+1}{j+1} - \tkj{k}{j} + \tkj{k}{j+1} 
\]
again by~\eqref{eq_xkj_AD}. If $\sigma_{k+1} = A$, then we have
\[
\begin{split}
\xkj{k+1}{j} - \xkj{k}{j} &= \tkj{k+1}{j} - \tkj{k+1}{j+1} + \skj{k+1}{j+1} + \lambda_{k+2,j+1} - \tkj{k}{j} + \tkj{k}{j+1}
\quad(\because~\text{Lemma~\ref{lemma_proof_Proposition_GC_2}}) \\
&= \skj{k}{j} + \lambda_{k+2,j+1} \quad(\because~\text{Lemma~\ref{lemma_proof_Proposition_GC_1}}) \\
&= \skj{k}{j} + \lambda_{i_{k,j}} \quad(\because~\eqref{eq_L_kj} \text{ and Lemma~\ref{lemma_proof_Proposition_GC_3}}) \\
& \geq 0 \quad(\because~\eqref{eq_string_skj})
\end{split}
\]
If $\sigma_{k+1} = D$, then we get
\[
\begin{split}
\xkj{k+1}{j} - \xkj{k}{j} &= \tkj{k+1}{k-j+1} - \tkj{k+1}{k-j+2} + \skj{k+1}{k-j+2} + \lambda_{k+2,j}- \tkj{k}{j} + \tkj{k}{j+1}
\quad(\because~\text{Lemma~\ref{lemma_proof_Proposition_GC_2}}) \\
&= \skj{k}{j} + \lambda_{k+2,j} \quad(\because~\text{Lemma~\ref{lemma_proof_Proposition_GC_1}}) \\
&= \skj{k}{j} + \lambda_{i_{k,j}} \quad(\because~\eqref{eq_L_kj} \text{ and Lemma~\ref{lemma_proof_Proposition_GC_3}}) \\
& \geq 0 \quad(\because~\eqref{eq_string_skj})
\end{split}
\]
This proves the inequality~\eqref{eq_GC_inequl_xkj} when $\sigma_k = A$.

Now we consider the case when $\sigma_k = D$. Again, we have the string cone inequalities  
\[
\tkj{k}{1} \geq \tkj{k}{2} \geq \cdots \geq \tkj{k}{k} \geq 0
\]
by Theorem~\ref{theorem_simplicial}. Then by~\eqref{eq_xkj_AD}, we have $\xkj{k+1}{j} \geq \xkj{k}{j}$. By considering $\xkj{k}{j} - \xkj{k+1}{j+1}$, we get
\[
\xkj{k}{j} - \xkj{k+1}{j+1} = \xkj{k+1}{j} - \tkj{k}{k-j+1} + \tkj{k}{k-j+2} - \xkj{k+1}{j+1}
\]
by~\eqref{eq_xkj_AD}. If $\sigma_{k+1} = A$, then we have
\[
\begin{split}
\xkj{k}{j} - \xkj{k+1}{j+1} &= \tkj{k+1}{j} - \tkj{k+1}{j+1} + \skj{k+1}{j+1} + \lambda_{k+2,j+1} - \tkj{k}{k-j+1} + \tkj{k}{k-j+2} \quad(\because~\text{Lemma~\ref{lemma_proof_Proposition_GC_2}}) \\
&= \skj{k}{k-j+1} + \lambda_{k+2,j+1} \quad(\because~\text{Lemma~\ref{lemma_proof_Proposition_GC_1}}) \\
&= \skj{k}{k-j+1} + \lambda_{i_{k,k-j+1}} \quad(\because~\eqref{eq_L_kj} \text{ and Lemma~\ref{lemma_proof_Proposition_GC_3}}) \\
& \geq 0 \quad(\because~\eqref{eq_string_skj})
\end{split}
\]
If $\sigma_{k+1} = D$, then we have
\[
\begin{split}
\xkj{k}{j} - \xkj{k+1}{j+1} &= \tkj{k+1}{k-j+1} - \tkj{k+1}{k-j+2} + \skj{k+1}{k-j+2} + \lambda_{k+2,j} - \tkj{k}{k-j+1} + \tkj{k}{k-j+2} \quad(\because~\text{Lemma~\ref{lemma_proof_Proposition_GC_2}}) \\
&= \skj{k}{k-j+1} + \lambda_{k+2,j}  \quad(\because~\text{Lemma~\ref{lemma_proof_Proposition_GC_1}}) \\
&= \skj{k}{k-j+1} + \lambda_{i_{k,k-j+1}} \quad(\because~\eqref{eq_L_kj} \text{ and Lemma~\ref{lemma_proof_Proposition_GC_3}}) \\
& \geq 0 \quad(\because~\eqref{eq_string_skj})
\end{split}
\]
Hence we proves the inequality~\eqref{eq_GC_inequl_xkj} when $\sigma_k = D$, which completes the proof. 
\end{proof}

\vspace{1cm}



\end{document}